\documentclass[reqno,a4paper, 11pt]{amsart}

\usepackage[utf8]{inputenc}

\usepackage[a4paper=true,pdfpagelabels]{hyperref}
\usepackage{graphicx}

\usepackage{amsfonts,epsfig}
\usepackage{latexsym}
\usepackage{mathabx}
\usepackage{amsmath}
\usepackage{amssymb}

\usepackage{color}

\newtheorem{theorem}{Theorem}
\newtheorem{lemma}[theorem]{Lemma}
\newtheorem{corollary}[theorem]{Corollary}
\newtheorem{proposition}[theorem]{Proposition}

\newtheorem{lettertheorem}{Theorem}
\newtheorem{letterlemma}[lettertheorem]{Lemma}

\theoremstyle{definition}

\theoremstyle{remark}

\numberwithin{equation}{section}

\newcommand{\set}[1]{\left\{#1\right\}}
\newcommand{\abs}[1]{\lvert#1\rvert}
\newcommand{\nm}[1]{\lVert#1\rVert}

\newcommand{\B}{\mathcal{B}}
\newcommand{\D}{\mathbb{D}}
\newcommand{\DD}{\widehat{\mathcal{D}}}
\newcommand{\Dd}{\widecheck{\mathcal{D}}}
\newcommand{\M}{\mathcal{M}}
\newcommand{\DDD}{\mathcal{D}}
\newcommand{\De}{\mathcal{D}}
\newcommand{\N}{\mathbb{N}}

\newcommand{\C}{\mathbb{C}}

\newcommand{\e}{\varepsilon}

\renewcommand{\phi}{\varphi}

\def\BMOA{\mathord{\rm BMOA}}
\DeclareMathOperator*{\esssup}{ess\,sup}

\newcommand{\hlp}{HL(p)}

\def\a{\alpha}       \def\b{\beta}        
       \def\De{{\Delta}}    \def\e{\varepsilon}
\def\la{\lambda}     \def\om{\omega}      
              
                  \def\z{\zeta}

\def\omg{\widehat{\omega}}

\def\fg{\widehat{f}}

\renewcommand{\H}{\mathcal{H}}

\newenvironment{Prf}{\noindent{\emph{Proof of}}}
{\hfill$\Box$ }

\newcommand{\Ho}{H_{\om}}

\makeatletter
\newcommand{\opnorm}{\@ifstar\@opnorms\@opnorm}
\newcommand{\@opnorms}[1]{%
  \left|\mkern-1.5mu\left|\mkern-1.5mu\left|
   #1
  \right|\mkern-1.5mu\right|\mkern-1.5mu\right|
}
\newcommand{\@opnorm}[2][]{%
  \mathopen{#1|\mkern-1.5mu#1|\mkern-1.5mu#1|}
  #2
  \mathclose{#1|\mkern-1.5mu#1|\mkern-1.5mu#1|}
}
\makeatother

\begin{document}

\title[Hilbert-type operator on Hardy spaces]{Hilbert-type operator induced by radial weight  on Hardy spaces}


\author{Noel Merchán}
\address{Departamento de Matemática Aplicada, Universidad de M\'alaga, Campus de
Teatinos, 29071 M\'alaga, Spain} \email{noel@uma.es}

\author{Jos\'e \'Angel Pel\'aez}
\address{Departamento de An\'alisis Matem\'atico, Universidad de M\'alaga, Campus de
Teatinos, 29071 M\'alaga, Spain} \email{japelaez@uma.es}

\author{Elena de la Rosa}
\address{Departamento de An\'alisis Matem\'atico, Universidad de M\'alaga, Campus de
Teatinos, 29071 M\'alaga, Spain} 
\email{elena.rosa@uma.es}

\thanks{This research was supported in part by Ministerio de Econom\'{\i}a y Competitividad, Spain, projects
PGC2018-096166-B-100; La Junta de Andaluc{\'i}a,
projects FQM210  and UMA18-FEDERJA-002.}

\subjclass[47G10, 30H10]{47G10, 30H10}

\maketitle

\begin{abstract}
 We consider 
 the Hilbert-type operator defined by
 $$
  H_{\omega}(f)(z)=\int_0^1 f(t)\left(\frac{1}{z}\int_0^z B^{\omega}_t(u)\,du\right)\,\omega(t)dt,$$
  where  $\{B^{\omega}_\zeta\}_{\zeta\in\mathbb{D}}$ are the  reproducing kernels of the Bergman space $A^2_\omega$   induced by a radial weight $\omega$
in the unit disc $\mathbb{D}$.
We prove that $H_{\omega}$ is bounded on the Hardy space $H^p$, $1<p<\infty$, if and only if
\begin{equation}
\label{abs1}
\sup_{0\le r<1} \frac{\widehat{\omega}(r)}{\widehat{\omega}\left( \frac{1+r}{2}\right)}<\infty, \tag{\dag}
\end{equation}
 and 
 \begin{equation*}
   \sup\limits_{0<r<1}\left(\int_0^r \frac{1}{\widehat{\omega}(t)^p} dt\right)^{\frac{1}{p}}
     \left(\int_r^1 \left(\frac{\widehat{\omega}(t)}{1-t}\right)^{p'}\,dt\right)^{\frac{1}{p'}} <\infty,
    \end{equation*}
    where $\widehat{\omega}(r)=\int_r^1 \omega(s)\,ds$. 

We also prove that $H_\omega: H^1\to H^1$ is bounded if and only if  \eqref{abs1} holds and   
 $$     \sup\limits_{r \in [0,1)} \frac{\widehat{\omega}(r)}{1-r} \left(\int_0^r \frac{ds}{\widehat{\omega}(s)}\right)<\infty.$$ 
        
As for the case $p=\infty$, $H_\omega$ is bounded from $H^\infty$  to $\BMOA$, or to the Bloch space, if and only if  \eqref{abs1} holds.

In addition, we prove that there does not exist radial weights $\omega$ such that $H_{\omega}: H^p \to H^p $, $1\le p<\infty$, is compact and we consider  the action of $H_{\omega}$ on some spaces of analytic functions closely related to Hardy spaces.  
  \end{abstract}
\section{Introduction}

For $0<p<\infty$, let $L^p_{ [0,1)}$ be the Lebesgue space of measurable functions such that 
$$\|f\|^p_{L^p_{[0,1)}}=\int_0^1 |f(t)|^p\,dt<\infty,$$ and
let $\H(\D)$  denote the space of analytic functions in the unit disc $\D=\{z\in\C:|z|<1\}$.
 The Hardy space $H^p$ consists of $f\in\H(\D)$ for which
 \begin{equation*}\label{normi}
    \|f\|_{H^p}=\sup_{0<r<1}M_p(r,f)<\infty,
   \end{equation*}
where
    $$
    M_p(r,f)=\left (\frac{1}{2\pi}\int_0^{2\pi}
    |f(re^{i\theta})|^p\,d\theta\right )^{\frac{1}{p}},\quad 0<p<\infty,
    $$
and 
    $$
    M_\infty(r,f)=\max_{0\le\theta\le2\pi}|f(re^{i\theta})|.
    $$
    For a  nonnegative function  $\om  \in L^1_{[0,1)}$, the extension to $\D$, defined by $\om(z)=\om(|z|)$ for all $z\in\D$, is called a radial weight.
 Let $A^2_{\om}$ denote the weighted Bergman space   of $f\in\H(\D)$ such that $\|f\|_{A^2_\omega}^2=\int_\D|f(z)|^2\omega(z)\,dA(z)<\infty$,
where $dA(z)=\frac{dx\,dy}{\pi}$ is the normalized area measure on $\D$.
 Throughout this paper we assume $\widehat{\om}(z)=\int_{|z|}^1\om(s)\,ds>0$ for all $z\in\D$, for otherwise $A^2_\om=\H(\D)$.

\vspace{1em}
\par
The Hilbert matrix is the infinite matrix whose entries are $h_{n,k}=(n+k+1)^{-1},$\, $k,n\in\N\cup\{0\}$. It
can be viewed as an operator on spaces of analytic functions, by its action on the Taylor coefficients
$$
\fg(k)\mapsto \sum_{k=0}^{\infty}
\frac{\fg(k)}{n+k+1}, \quad n\in\N\cup\{0\},
$$
called
the Hilbert operator. 
That is, if
 $f(z)=\sum_{k=0}^\infty \fg(k)z^k\in \H(\D) $
\begin{equation}\label{H}
H(f)(z)=
\sum_{n=0}^{\infty}\left(\sum_{k=0}^{\infty}
\frac{\fg(k)}{n+k+1}\right)z^n,
\end{equation}
whenever the  right hand side makes sense and defines an
analytic function in $\D $.

The Hilbert operator $H$ is bounded on  Hardy spaces 
$H^p$ if and only if $1<p<\infty$ \cite{DiS}. A proof of this result can be obtained 
using the following integral representation, valid for any $f\in H^1$, 
\begin{equation}\label{H-int}
H(f)(z)=\int_0^1f(t)\frac{1}{1-tz}\,dt.
\end{equation}
Going further, the formula \eqref{H-int}
has been  employed to solve  a good number of questions in operator theory related to  the boundedness,  the operator norm  and  the spectrum of the Hilbert operator on classical  spaces of analytic functions \cite{AlMonSa,Di,DJV,PelRathg}.
During the last decades several generalizations of the Hilbert operator have attracted a considerable amount
of attention  \cite{GaPe2010,GirMerIE2017,PelRathg,PelRosa21}. We will focus on the following, introduced in \cite{PelRosa21}.
For a radial weight $\omega$, we consider the Hilbert-type operator
\begin{equation*}\label{eq:i1}
    H_{\omega}(f)(z)=\int_0^1 f(t)\left(\frac{1}{z}\int_0^z B^{\om}_t(\z)d\z\right)\,\om(t)dt,
\end{equation*}
where $\{B^\om_z\}_{z\in\D}\subset A^2_\om$ are the Bergman reproducing kernels of $A^2_\omega$.
The choice $\om=1$ gives the integral representation \eqref{H-int}
 of the classical Hilbert operator, therefore it is natural to think of the features of a radial weight $\omega$
  so that $H_\omega$ has some of the nice properties of the (classical) Hilbert operator. 
  In this paper, among other results,
   we describe the radial weights $\omega$ such that the Hilbert-type operator $H_\omega$ is bounded on $H^p$, $1\le p<\infty$.
\vspace{1em}

In order to state our results some more notation is needed.
 For $0<p<\infty$, the Dirichlet-type space $D^p_{p-1}$ is the space
of $f\in\H(\D)$ such that
$$\| f\|^p_{D^p_{p-1}}=|f(0)|^p+\int_\D |f'(z)|^p(1-|z|)^{p-1}\,dA(z)<\infty,$$
 and the Hardy-Littlewood space $HL(p)$ consists of the $f(z)=\sum\limits_{n=0}^{\infty} \widehat{f}(n) z^n\in \H(\D)$ such that
$$ \nm{f}^p_{HL(p)}=\sum\limits_{n=0}^{\infty} \abs{\widehat{f}(n)}^p (n+1)^{p-2}<\infty.$$
We will also consider  the space $H(\infty,p)=\{f\in\H(\D): \|f\|^p_{H(\infty,p)}=\int_0^1 M^p_\infty(r,f)\,dr<\infty\}.$
These spaces satisfy the well-known inclusions          
\begin{equation}\label{eq:HpDpp<2}
D^p_{p-1}\subset H^p\subset HL(p),\quad 0<p\le 2,
\end{equation}
\begin{equation}\label{eq:HpDpp>2}
HL(p)\subset H^p\subset D^p_{p-1} ,\quad 2\le p<\infty,
\end{equation}
and
\begin{equation}\label{Xp:Hinftyp}
H^p\subset H(\infty,p), \quad D^p_{p-1}\subset H(\infty,p),\quad 0<p<\infty.
\end{equation}
See \cite{Duren,Flett,LP} for  proofs of \eqref{eq:HpDpp<2} and \eqref{eq:HpDpp>2}, and
\cite[p. 127]{Pomm} and \cite[Lemma~4]{GaGiPeSis} for a  proof of \eqref{Xp:Hinftyp}.

The   Bergman reproducing kernels,  induced by a radial weight $\om$, can be written as 
 $B^\om_z(\z)=\sum \overline{e_n(z)}e_n(\z)$  for each orthonormal basis $\{e_n\}$ of $A^2_\om$, and therefore 
 using the basis
 induced by the normalized monomials,
\begin{equation}\label{eq:B}
B^\om_z(\z)=\sum_{n=0}^\infty\frac{\left(\overline{z}\z\right)^n}{2\om_{2n+1}}, \quad z,\z\in \D.
\end{equation}
 Here $\om_{2n+1}$ are the odd moments of $\om$, and in general from now on we write $\om_x=\int_0^1r^x\om(r)\,dr$ for all $x\ge0$. A radial weight $\omega$
 belongs to the class~$\DD$ if
 $\widehat{\om}(r)\le
C\widehat{\om}(\frac{1+r}{2})$ for some constant $C=C(\omega)>1$ and all $0\le r <1$. If there exist $K=K(\om)>1$ and $C=C(\om)>1$ such that $\widehat{\om}(r)\ge C\widehat{\om}\left(1-\frac{1-r}{K}\right)$ for all $0\le r<1$, then $\om\in\Dd$. Further, we write $\DDD=\DD\cap\Dd$ for short. Recall that $\om\in\M$ if there exist constants $C=C(\om)>1$ and $K=K(\om)>1$ such that $\om_{x}\ge C\om_{Kx}$ for all $x\ge1$. It is known that $\Dd\subset \M$ \cite[Proof of Theorem~3]{PR19} but $\Dd \subsetneq \M$
 \cite[Proposition~14]{PR19}. However, \cite[Theorem~3]{PR19} ensures that $\DDD=\DD\cap\Dd=\DD\cap\M$.
 These classes of weights arise in meaningful questions concerning radial weights and classical operators, such as the differentiation operator $f^{(n)}$ or 
 the Bergman projection $ P_\om(f)(z)=\int_{\D}f(\z) \overline{B^\om_{z}(\z)}\,\om(\z)dA(\z)$ \cite{PR19}.
We will also deal with the sublinear   Hilbert-type operator 
\begin{equation*}\label{eq:sub}
    \widetilde{H_{\omega}}(f)(z)=\int_0^1 |f(t)|\left(\frac{1}{z}\int_0^z B^{\om}_t(\z)d\z\right)\,\om(t)\,dt.
\end{equation*}

If $X,Y \subset \H(\D)$ are normed vector spaces, and $T$ is a sublinear operator,
 we denote
$\|T\|_{X\to Y}=\sup_{\|f\|_X\le 1}\| T(f)\|_{Y}$.

\begin{theorem}
\label{bounded Xp}
Let  $\om$ be a radial weight and $1<p<\infty$. Let $X_p,Y_p\in \{H(\infty,p),H^p, D^p_{p-1}, HL(p)\}$ and 
 $T\in\{ \Ho, \widetilde{\Ho}\}$. Then the following statements are equivalent:
\begin{itemize}
\item[(i)] $T\,:\,X_p \to Y_p$ is bounded; 
\item[(ii)] $\om \in \DDD$ and 
$
M_p(\om)=\sup\limits_{N\in \N} \left(\sum\limits_{n=0}^N \frac{1}{(n+1)^2 \om_{2n+1}^p}\right)^{\frac{1}{p}}\left(\sum\limits_{n=N}^{\infty} \om_{2n+1}^{p'}(n+1)^{p'-2}\right)^{\frac{1}{p'}}<\infty;$
\item[(iii)] $\om \in \DD$ and 
$
M_p(\om)=\sup\limits_{N\in \N} \left(\sum\limits_{n=0}^N \frac{1}{(n+1)^2 \om_{2n+1}^p}\right)^{\frac{1}{p}}\left(\sum\limits_{n=N}^{\infty} \om_{2n+1}^{p'}(n+1)^{p'-2}\right)^{\frac{1}{p'}}<\infty;
$
\item[(iv)] $\om \in \DD$ and 
$
M_{p,c}(\om)=
\sup\limits_{0<r<1} \left(\int_0^r \frac{1}{\omg(t)^p} dt\right)^{\frac{1}{p}}
\left(\int_r^1 \left(\frac{\omg(t)}{1-t}\right)^{p'}\,dt\right)^{\frac{1}{p'}}<\infty$.
\end{itemize}

\end{theorem}

 \par\vspace{1em}

The proof of (i)$\Rightarrow$(iii) of Theorem~\ref{bounded Xp} has two steps. Firstly, we prove that $\om\in\DD$, and later on  the condition $M_p(\om)<\infty$ is obtained by using polynomials of the form $f_{N, M}(z)=\sum\limits_{k=N}^M \om_{2k}^{\alpha} (k+1)^{\beta}z^k,\,  N,M\in\N,\; \a, \b \in \mathbb{R}$ as test functions. Then, we see that any radial weight $\omega$ satisfying the condition  $M_p(\om)<\infty$, belongs to $\mathcal{M}$. This proves (ii)$\Leftrightarrow$(iii). The proof of (iii)$\Leftrightarrow$(iv) is a calculation based on known descriptions of the class $\DD$ \cite[Lemma~2.1]{PelSum14}. Finally, 
we prove  (iv)$\Rightarrow$(i) which is the most involved implication in the proof of Theorem~\ref{bounded Xp}.
In order to obtain it, we merge thecniques coming from complex and harmonic analysis, such as a very convenient description of the class $\DDD$, see Lemma~\ref{le:descriptionDDD} below, precise estimates of 
the integral means of order $p$ of the derivative of the kernels 
$K^\omega_u(z)=\frac{1}{z}\int_0^z B^\omega_u(z)\,du $, decomposition norm theorems and classical weighted inequalities for Hardy operators. 

Observe that both, the discrete condition $M_p(\om)<\infty$ and its continuous version $M_{p,c}(\om)<\infty$, are used in the proof
of Theorem~\ref{bounded Xp}. The first one 
follows  from  (i), and   the condition $M_{p,c}(\om)<\infty$ is employed to prove that
$T\,:\,X_p \to Y_p$ is bounded.

\vspace{1em}
As for the  case $p=1$ we obtain the following result.

\begin{theorem}
\label{bounded X_1}
Let  $\om$ be a radial weight, $X_1,Y_1\in \{H(\infty,1),H^1, D^1_{0}, HL(1)\}$ and   
 $T\in\{ \Ho, \widetilde{\Ho}\}$.  Then the following statements are equivalent:
\begin{itemize}
    \item[(i)] $T: X_1\to Y_1$ is bounded; 
    \item[(ii)] $\om \in \DD$ and the measure $\mu_\om$ defined as $d\mu_\om(z)= \om(z)\left(\int_0^{|z|} \frac{ds}{\omg(s)}\right)\,\chi_{[0,1)}(z)\,  dA(z)$ is a $1$-Carleson measure for $X_1$;
    \item[(iii)]  $\om \in \DD$ and satisfies the condition 
    \begin{equation*}
        \label{m1}
      M_{1,c}(\om)= \sup\limits_{a \in [0,1)}
        \frac{1}{1-a}\int_a^1 \om(t)\left(\int_0^t \frac{ds}{\omg(s)}\right)\,dt<\infty\textit{;}
        \end{equation*}
       \item[(iv)]  $\om \in \DDD$ and satisfies the condition   $ M_{1,c}(\om)<\infty$;
        \item[(v)]  $\om \in \DD$ and satisfies the condition   
        $M_{1,d}(\om)= \sup\limits_{a \in [0,1)} \frac{\omg(a)}{1-a} \left(\int_0^a \frac{ds}{\omg(s)}\right)<\infty$;
 \item[(vi)]  $\om \in \DD$ and satisfies the condition   
 \begin{equation*}
        \label{M1}  
 M_1(\om)= \sup \limits_{N \in \N} (N+1)\om_{2N}\sum\limits_{k=0}^{N}\frac{1}{(k+1)^2 \om_{2k}}<\infty .  
      \end{equation*}  
\end{itemize}

\end{theorem}

We recall that
 given a Banach space (or a complete metric space) $X$ of analytic functions on $\D$, a positive Borel measure
$\mu$ on $\D$ is called a $q$-Carleson measure for $X$
 if the identity 
operator $I_d:\, X\to L^q(\mu)$\index{$I_d$} is bounded. 
Carleson provided a geometric description of $p$-Carleson measures for Hardy spaces $H^p$, $0<p<\infty$,  
\cite[Chapter~9]{Duren}. These measures are called classical Carleson measures.
The proof of Theorem~\ref{bounded X_1} uses characterizations of Carleson measures for $X_1$-spaces,  universal Ces\`aro basis of polynomials and  some of the main ingredients  of the proofs of Theorem~\ref{bounded Xp} and \cite[Theorem~2]{PelRosa21}. 

Concerning the classes of radial weights $\DD$ and $M_{p,c}=\{ \omega: M_{p,c}(\omega)<\infty\}$, $1\le p<\infty$,
a standard weight, $\om (z)=(1-|z|)^{\b}$, $\b >-1$, satisfies the condition  $M_{p,c}(\omega)<\infty$ if and only if  $\b >\frac{1}{p}-1$, so 
$H_\omega: H^p\to H^p$ is bounded if and only if $\b >\frac{1}{p}-1$. Moreover, a calculation shows that  the exponential type weight
$\omega(r)=\exp\left( -\frac{1}{1-r}\right)\in M_{p,c}$ for any $p\in [1,\infty)$, but $\omega\notin\DD$, see 
 \cite[Example 3.2]{Sis} for further details. So, $\DD$ and $M_{p,c}$ are not included in each other.

\vspace{1em}
The  study of the radial weights $\omega$ such that $H_\omega: H^p\to H^p$ is bounded, has been previously considered in \cite{PelRosa21}. Indeed,
Theorem~\ref{bounded X_1} improves \cite[Theorem~2]{PelRosa21},  by removing the initial hypothesis $\omega\in\DD$. On the other hand, 
 \cite[Theorem~3]{PelRosa21} describes the weights $\omega\in\DD$ such that $H_\omega: L^p_{[0,1)}\to H^p$ is bounded, and consequently gives 
a sufficient condition for the boundedness of $H_\omega: H^p\to H^p$, $1<p<\infty$.
 The following improvement of \cite[Theorem~3]{PelRosa21} is a byproduct of Theorem~\ref{bounded Xp}.

\begin{corollary}\label{th:lphp}
Let  $\om$ be a radial weight and $1<p<\infty$. Let $Y_p\in \{H(\infty,p),H^p, D^p_{p-1}, HL(p)\}$ and 
 $T\in\{ \Ho, \widetilde{\Ho}\}$. Then the following statements are equivalent:
\begin{itemize}
    \item[(i)] $T:L^p_{[0,1)} \to Y_p$ is bounded;
    \item[(ii)]  $\omega\in \DDD$ and satisfies the condition
    \begin{equation*}
     m_p(\omega)=\sup\limits_{0<r<1}\left(1+\int_0^r \frac{1}{\omg(t)^p} dt\right)^{\frac{1}{p}}
     \left(\int_r^1 \om(t)^{p'}\,dt\right)^{\frac{1}{p'}} <\infty;  
    \end{equation*}
     \item[(iii)]  $\omega\in \DD$ and satisfies the condition $ m_p(\omega) <\infty$.  
    \end{itemize}
\end{corollary}

In relation to an analogous result to Corollary~\ref{th:lphp} for $p=1$, Theorem~\ref{bounded L1} below shows that 
the radial weights such that $T:L^1_{[0,1)} \to Y_1$ is bounded, where  $Y_1\in \{H(\infty,1),H^1, D^1_{0}, HL(1)\}$ and 
 $T\in\{ \Ho, \widetilde{\Ho}\}$,  are the weights $\omega\in\DDD$ such that
 $m_1(\om)= \esssup_{t \in [0,1)}
         \om(t)\left(1+\int_0^t \frac{ds}{\omg(s)}\right)<\infty.$ 
         
In view of the above findings, we compare the conditions $M_{p,c}(\om)<\infty$, $M_{p,d}(\om)<\infty$ and  $m_p(\om)<\infty$
in order to put the boundedness of  $T\,:\,X_p \to Y_p$ alongside the boundedness of $T:L^p_{[0,1)} \to Y_p$,
where $X_p, Y_p\in \{H(\infty,p),H^p, D^p_{p-1}, HL(p)\}$ and 
 $T\in\{ \Ho, \widetilde{\Ho}\}$ for $1\le p<\infty$. Bearing in mind
 \eqref{Xp:Hinftyp}, it is clear that the condition  $m_p(\om)<\infty$ implies that $M_{p,c}(\om)<\infty$, for any weight $\omega\in\DD$.
Moreover, observe that $M_{p,c}(\omega)<\infty$ if and only if
$$
\sup\limits_{0<r<1} \left(1+\int_0^r \frac{1}{\omg(t)^p} dt\right)^{\frac{1}{p}}
\left(\int_r^1 \left(\frac{\omg(t)}{1-t}\right)^{p'}\,dt\right)^{\frac{1}{p'}}<\infty,\quad\text{ when $1<p<\infty$}$$ and
$\sup\limits_{a \in [0,1)} \frac{\omg(a)}{1-a} \left(1+\int_0^a \frac{ds}{\omg(s)}\right)<\infty$ if and only if $M_{1,d}(\om)<\infty$. So, the conditions $M_{p,c}(\om)<\infty$ and $m_p(\om)<\infty$, 
 are equivalent for any $1\leq p < \infty$ whenever $\omega$ satisfies 
 the pointwise inequality
\begin{equation}\label{eq:regularityintro} 
  \om(t)\lesssim \frac{\widehat{\omega}(t)}{1-t},\quad t \in[0,1),
\end{equation}  
  and $\om \in \DD$. 
  The condition \eqref{eq:regularityintro} implies 
restrictions on the decay and  on the regularity of the weight, in fact if $\omega$ fulfills \eqref{eq:regularityintro} then $\omega$  cannot decrease rapidly and cannot oscillate strongly.
For instance, the exponential type weight
$\omega(r)=\exp\left( -\frac{1}{1-r}\right)$, 
which is a prototype of rapidly decreasing weight  (see \cite{PP}),
has the property  
$$\widehat{\omega}(r)\asymp  \omega(r) (1-r)^2,\quad 0\le r<1,$$
 so it does not satisfy   \eqref{eq:regularityintro}.
 On the other hand,  
  any regular or rapidly increasing weight satisfies \eqref{eq:regularityintro}. 
  Regular and rapidly increasing weights are large subclasses of $\DD$, see \cite[Section~1.2]{PR} for the definitions and examples of these classes of radial weights. However, 
 we  construct in Corollaries~\ref{co:comparisonLp-Xp} and \ref{co:comparisonL1-X1} weights $\omega\in\DDD$ 
   with a strong oscillatory behaviour so that
   $M_{p,c}(\om)<\infty$ and $m_p(\om)=\infty$, and consequently they do not satisfy \eqref{eq:regularityintro}.

\vspace{1em} 
With the aim of discussing some results concerning the  case $p=\infty$, we recall that the space $\BMOA$
consists of those functions in the Hardy space~$H^1$ that have
bounded mean oscillation
on the boundary of  $\D$~\cite{GiBMO}, and the Bloch space $\B$ is the space of all analytic functions
on $\D$ such that
\begin{displaymath}
\|f\|_{\mathcal{B}}=|f(0)|+\sup_{z\in\D}(1-|z|^2)\,|f'(z)|<\infty.
\end{displaymath}
We also
 consider the space $HL(\infty)$ of the $f(z)=\sum_{n=0}^\infty \fg(n) z^n \in\H(\D)$ such that
 $$\|f\|_{HL(\infty)}=\sup_{n\in\N\cup\{0\}}(n+1)\left| \fg(n)\right|<\infty.$$
The following chain of inclusions hold \cite{GiBMO}
\begin{equation}\label{eq:inftyintro}
HL(\infty)\subsetneq \BMOA \subsetneq \B.
\end{equation}

If $\om$ is a radial weight 
\begin{align*}\label{eq:noHinfty}
  H_{\omega}(1)(x)
  =\sum\limits_{n=0}^{\infty}\frac{\om_n}{2\om_{2n+1}(n+1)}x^n
  &\geq \frac{1}{2x}\sum\limits_{n=0}^{\infty}\frac{x^{n+1}}{n+1}=\frac{1}{2x}\log\left(\frac{1}{1-x}\right), \quad x\in (0,1),
\end{align*}
so  $H_{\omega}$ is not bounded on $H^\infty$.  As for the classical Hilbert matrix $H$, it is bounded from $H^\infty$ to $BMOA$ \cite[Theorem 1.2]{LaNoPa2012}.
So, it is natural wondering about the radial weights such that $H_{\om}:H^{\infty}\to \BMOA$ is  bounded.  The next result answers this question.

\begin{theorem}
\label{th:H infty Bloch intro}
Let $\om $ be a radial weight and 
let $T\in\{\Ho,\widetilde{\Ho}\}$. Then, the following statements are equivalent:
\begin{itemize}
\item[(i)] $T: H^{\infty}\to HL(\infty)$ is bounded;
\item[(ii)] $T :H^{\infty}\to \BMOA$ is bounded;
\item[(iii)]  $T: H^{\infty}\to \B$ is bounded;
\item [(iv)] $\om \in \DD$.
\end{itemize}
\end{theorem}
The equivalence (iii)$\Leftrightarrow$(iv) was proved in \cite[Theorem~1]{PelRosa21}, so our contribution in Theorem~\ref{th:H infty Bloch intro} consists on proving the rest of equivalences.

Bearing in mind Theorems~\ref{bounded Xp}, \ref{bounded X_1} and \ref{th:H infty Bloch intro}, we deduce that 
 $T\in\{\Ho,\widetilde{\Ho}\}$ is bounded from $H^\infty$ to $HL(\infty)$  if 
 $T: X_p\to Y_p$ is bounded, where
 $X_p,Y_p\in \{H(\infty,p),H^p, D^p_{p-1}, HL(p)\}$, $1\le p<\infty$. We prove that this is a general phenomenon for Hilbert-type operators and parameters 
 $1\le q<p$.

\begin{theorem}\label{th:q<p}
Let $\om$ be radial weight, $T\in\{H_\omega,\widetilde{H_\omega}\}$ and $1\le q<p< \infty$.  Further, let $X_q,Y_q\in \{H^q, D^q_{q-1}, HL(q), H(\infty,q)\}$ and
$X_p,Y_p\in \{H^p, D^p_{p-1}, HL(p), H(\infty,p)\}$.
 If $T: X_q\to Y_q$ is bounded, then $T: X_p\to Y_p$
is bounded.
\end{theorem}

We also prove that that there does not exist radial weights $\omega$ such that $H_\omega: X_p\to Y_p$
is compact, where  $X_p,Y_p\in \{H^p, D^p_{p-1}, HL(p), H(\infty,p)\}$ and $1\le p< \infty$, neither radial weights such that $H_{\omega}:H^{\infty} \to \B$ is compact, see Theorems \ref{pr:nocompactxp}, \ref{pr:nocompactx1}, \ref{th:nocompactoHinfty} below.

The letter $C=C(\cdot)$ will denote an absolute constant whose value depends on the parameters indicated
in the parenthesis, and may change from one occurrence to another.
We will use the notation $a\lesssim b$ if there exists a constant
$C=C(\cdot)>0$ such that $a\le Cb$, and $a\gtrsim b$ is understood
in an analogous manner. In particular, if $a\lesssim b$ and
$a\gtrsim b$, then we write $a\asymp b$ and say that $a$ and $b$ are comparable.
We remark that if $a$ or $b$ are quantities which depends on a radial weight $\omega$,
 the constant $C$ such that  $a\lesssim b$ or  $a\gtrsim b$ may depend on $\omega$ but it does not depend on $a$ neither on $b$.

\vspace{1em}
The rest of the paper is organized as follows. Section~\ref{sec:pre} is devoted to prove some auxiliary results. 
 We prove   Theorem~\ref{bounded Xp} and Corollary \ref{th:lphp} in Section~\ref{sec:p>1},  and
Theorem~\ref{bounded X_1} is proved  in Section~\ref{sec:p=1}. 
Section~\ref{sec:infty} contains a proof of
Theorem~\ref{th:H infty Bloch intro} 
and Theorem~\ref{th:q<p} is proved in Section~\ref{sec:q<p} together with some reformulations of the condition $M_{p,c}(\omega)<\infty$.

\section{Preliminary results}\label{sec:pre}

In this section, we will prove some convenient preliminary results which will be repeatedly used throughout the rest of the paper.   
The first auxiliary lemma contains several characterizations of upper  doubling radial weights. For a proof, see \cite[Lemma~2.1]{PelSum14}.

\begin{letterlemma}
\label{caract. pesos doblantes}
Let $\om$ be a radial weight on $\D$. Then, the following statements are equivalent:
\begin{itemize}
    \item[(i)] $\om \in \DD$;
    \item[(ii)] There exist $C=C(\om)\geq 1$ and $\b_0=\b_0(\om)>0$ such that
    $$ \omg(r)\leq C \left(\frac{1-r}{1-t}\right)^{\b}\omg(t), \quad 0\leq r\leq t<1;$$
   for all $\b\geq \b_0$.
   \item[(iii)] $$ \int_0^1 s^x \om (s) ds\asymp \omg\left(1-\frac{1}{x}\right),\quad x \in [1,\infty);$$
   \item[(iv)] There exists $C=C(\om)>0$ and $\b=\b(\om)>0$ such that 
   $$ \om_x\leq C \left(\frac{y}{x}\right)^{\b}\om_y,\quad 0<x\leq y<\infty ;$$
     \item[(v)] $ \DD(\om)=\sup\limits_{n\in \N}\frac{\om_n}{\om_{2n}}<\infty .$
\end{itemize}
\end{letterlemma}

We will also use the following characterizations of the class $\Dd$, 
see \cite[(2.27)]{PR19}.
\begin{letterlemma}
\label{caract. D check}
Let $\om$ be a radial weight. The following statements are equivalent:
\begin{itemize}
\item[(i)] $\om	\in \Dd$;
\item[(ii)] There exist $C=C(\om)>0$ and $\a_0=\a_0(\om)>0$ such that
$$\omg(s)\leq C \left(\frac{1-s}{1-t}\right)^{\a}\omg(t), \quad 0\leq  t\leq s<1$$
for all $0<\alpha\le \alpha_0$;
\item[(iii)]There exist $K=K(\om)>1 $ and $C=C(\om)>0 $ such that
\begin{equation}
\label{D chek y k}
\int_r^{1-\frac{1-r}{K}}\om (s)ds \geq C \omg(r), \quad 0\leq r<1.
\end{equation}
\end{itemize}
\end{letterlemma}

Embedding relations among spaces $X_p,Y_p\in \{H^p, D^p_{p-1}, HL(p), H(\infty,p)\}$ are quite useful in the study of 
operators acting on them. In particular,  we recall that
\begin{equation}\label{fejer}
\|f\|_{H(\infty,p)}\le C_p \|f\|_{X_p},\quad 0< p<\infty,
\end{equation}
for $X_p\in\{ H^p,D^p_{p-1}\}$, see \cite[p. 127]{Pomm} and \cite[Lemma~4]{GaGiPeSis}. 

This inequality is no longer true for $X_p=HL(p)$ if $0<p<1$.  In fact, take 
$f(z)=\sum_{n=0}^\infty 2^{\frac{n}{p}}z^{2^n}$.
A calculation shows that $f\in HL(p)$,
 if $0<p<1$. However, using \cite[Theorem~1]{MatPavPAMS83}, 
$$\|f\|^p_{H(\infty,p)}=  \int_0^1 \left(\sum_{n=0}^\infty 2^{\frac{n}{p}}s^{2^n}\right)^p\,ds
\asymp \sum_{n=0}^\infty 1=\infty.$$

Our following result extends the inequality \eqref{fejer}  to $X_p=HL(p)$ and $1\le p<\infty$.

\begin{lemma}
\label{HLP C LP}
Let $1\le p<\infty$. Then, there is $C_p>0$ such that
$$\|f\|_{H(\infty,p)}\le C_p \|f\|_{X_p},\quad f\in\H(\D),$$
where $X_p\in\{ H^p,D^p_{p-1}, HL(p)\}$.
\end{lemma}
\begin{proof}
By \eqref{fejer} it is enough to prove the inequality for $X_p=\hlp$.
 By \cite[Theorem~1]{MatPavPAMS83} and  H\"older's inequality 
\begin{align*}
\int_0^1 M^p_\infty(t,f)\,dt & \le \int_{0}^1 \left(\sum_{n=0}^\infty |\widehat{f}(n)|t^n\right)^p\,dt
\\ &\lesssim |\widehat{f}(0)|^p+ \sum_{n=0}^\infty 2^{-n}\left(\sum_{k=2^n}^{2^{n+1}-1} |\widehat{f}(k)|\right)^p
\\ & \le |\widehat{f}(0)|^p+\sum_{n=0}^\infty 2^{n(p-2)}\sum_{k=2^n}^{2^{n+1}-1} |\widehat{f}(k)|^p
\\ &\lesssim |\widehat{f}(0)|^p+\sum_{n=0}^\infty \sum_{k=2^n}^{2^{n+1}-1} (k+1)^{p-2}|\widehat{f}(k)|^p
= \|f\|_{\hlp}^p.
\end{align*}
This finishes the proof.
\end{proof}

For $0<p<\infty$ and $\omega$ a radial weight, let $L^p_{\omega, [0,1)}$ be the Lebesgue space of measurable functions such that 
$$\|f\|^p_{L^p_{\omega, [0,1)}}=\int_0^1 |f(t)|^p\omega(t)\,dt<\infty.$$
Next, we will prove that the sublinear operator $\widetilde{H_{\omega}}$ does not distinguish the norm of the spaces
$H(\infty,p), HL(p), D^p_{p-1}, H^p,$
when $1<p<\infty$ and $\om \in \DD$.

\begin{lemma}\label{le:Htildeequiv}
 Let $\om \in \DD$, $1<p<\infty$ and $X_p, Y_p\in \{H(\infty,p), HL(p), D^p_{p-1}, H^p\}$. Then, 
$$  \nm{\widetilde{H_{\omega}} (f)}_{X_p}\asymp \nm{\widetilde{H_{\omega}} (f)}_{Y_p},\, \quad f \in  L^1_{\om,[0,1)}.$$ 
\end{lemma}
\begin{proof}
Here and on the following, let us denote 
$I(n)=\{k\in\N: 2^n\le k<2^{n+1}\}$,\, $n\in \N\cup\{0\}$.  
By Lemma~\ref{caract. pesos doblantes}
\begin{align}\label{eq:omegaIn}
 \om_{2^{n+2}}\asymp \omega_{2k+2}\asymp \omega_{2k}\asymp \om_{2^{n}},\quad\text{ for any  $n\in \N\cup\{0\}$ and $k\in I(n)$}.
 \end{align}
The above equivalences and  \cite[Theorem~1]{MatPavPAMS83}, yield
\begin{equation*}\begin{split}
 \nm{\widetilde{H_{\omega}} (f)}_{H(\infty,p)}^p
& \asymp \sum_{n=0}^\infty 2^{-n}\left( \sum_{k\in I(n)} \frac{\int_0^1 |f(t)|t^k\omega(t)\,dt}{(k+1)\om_{2k+1}} \right)^p+
 \left( \int_0^1 |f(t)|\omega(t)\,dt\right)^p
\\ &  \asymp \sum_{n=0}^\infty 2^{-n}
\left(  \frac{\int_0^1 |f(t)|t^{2^n}\omega(t)\,dt}{\om_{2^{n+1}}} \right)^p+
 \left( \int_0^1 |f(t)|\omega(t)\,dt\right)^p
\\ &  \asymp  \nm{\widetilde{H_{\omega}} (f)}_{HL(p)}^p, \quad f \in L^1_{\om,[0,1)}.
\end{split}\end{equation*}
This, together with \cite[Lemma~8]{PelRosa21}, finishes the proof.
\end{proof}

\section{Hilbert-type operators acting on $X_p$-spaces, $1<p<\infty$.}\label{sec:p>1}

\subsection{Necessity part of Theorem~\ref{bounded Xp}.}

We begin this section with the  construction of appropriate families of  test functions to be used in the proof of Theorem~\ref{bounded Xp}.
To do this, some notation and previous results are needed. 
Let  $g(z)=\sum\limits_{k=0}^{\infty} \widehat{g}(k) z^k \in \H(\D)$, and
denote
$\De_n g(z) =\sum\limits_{k \in I(n)} \widehat{g}(k)z^k$.    
In the particular case $g(z)=\frac{1}{1-z}$, we simply write 
$\De_n(z)=\De_n(g)(z) =\sum\limits_{k \in I(n)}z^k$. We recall that  
\begin{equation}\label{eq:Dn}
 \| \De_n\|_{H^p}\asymp 2^{n(1-1/p)}, \quad n\in \N\cup\{0\},\quad 1<p<\infty,
 \end{equation}
see \cite[Lemma~2.7]{CPPR}.

 For any $n_1, \,n_2 \in \N\cup\set{0}$, $n_1<n_2$, write
$ S_{n_1,n_2} g(z)=\sum\limits_{k=n_1}^{n_2-1} \widehat{g}(k) z^k$.
The next known result can be proved mimicking the proof of \cite[Lemma~3.4]{LaNoPa2012} (see also \cite[Lemma~E]{PelRathg}), that is,
    by summing by parts and using the M. Riesz projection
theorem.
 
\begin{letterlemma}
\label{lema de los lambda}
Let $1<p<\infty$ and $\la=\set{\la_k}_{k=0}^{\infty}$ be a positive and monotone sequence. Let $ g(z)=\sum\limits_{k=0}^{\infty}b_k z^k$ and $(\la g)(z)=\sum\limits_{k=0}^{\infty}\la_k b_k z^k$.
\begin{itemize}
    \item[(a)] If $\set{\la_k}_{k=0}^{\infty}$ is nondecreasing, there exists a constant $C>0$ such that
    $$ C^{-1}\la_{n_1}\nm{S_{n_1,n_2}g}_{H^p}\leq \nm{S_{n_1,n_2}(\la g)}_{H^p}\leq C \la_{n_2}\nm{S_{n_1,n_2}g}_{H^p}.$$
    \item[(b)] If $\set{\la_k}_{k=0}^{\infty}$ is nonincreasing, there exists a constant $C>0$ such that
    $$ C^{-1}\la_{n_2}\nm{S_{n_1,n_2}g}_{H^p}\leq \nm{S_{n_1,n_2}(\la g)}_{H^p}\leq C \la_{n_1}\nm{S_{n_1,n_2}g}_{H^p}.$$
\end{itemize}
\end{letterlemma}

\begin{lemma}
\label{lema f test Mp}  
Let  $\omega\in\DD$, $1<p<\infty$, $\alpha,\beta\in\mathbb{R}$ and $M, N \in \N\cup\{0\}$ such that $0\le N<4N+1\le M$. Let us 
 consider the function 
\begin{equation*}
f_{N,M}(z)=\sum\limits_{k=N}^M \om_{2k}^{\alpha} (k+1)^{\beta}z^k.
\end{equation*}
Then,
\begin{equation}\label{eq:mpr0}
 \|f_{ N, M}\|_{HL(p)}\asymp  \|f_{ N, M}\|_{H^p}\asymp  \|f_{ N, M}\|_{D^p_{p-1}},
 \end{equation}
where the constants involved do not depend on $M$ or $N$.
In particular, if $\alpha=0$ then \eqref{eq:mpr0} holds for any radial weight. 
\end{lemma}

\begin{proof}
Firstly,  let us show that for all $N, M \in \N$, $M>N$, 
\begin{equation}\label{eq:mpr1}
\|f_{ 2^N +1, 2^M }\|_{D^p_{p-1}}\asymp \|f_{2^N+1, 2^M}\|_{HL(p)}.
\end{equation}

\cite[Theorem~2.1(b)]{MatPav} (see also \cite[7.5.8]{Pabook}), Lemma \ref{lema de los lambda}, \eqref{eq:omegaIn}
 and \eqref{eq:Dn} implies
\begin{align*}
\|f_{2^N+1, 2^M}\|_{D^p_{p-1}}^p &\asymp\sum\limits_{n=N}^{M-1} 2^{-np} \left\|\sum\limits_{k\in I(n)} \om_{2k+2}^{\alpha} (k+2)^{\beta} (k+1) z^k\right\|_{H^p}^p 
\\ & \asymp \sum\limits_{n=N}^{M-1} 2^{np\beta} \om_{2^{n+1}}^{p\alpha}\|\Delta_n\|^p_{H^p}
\\ & \asymp \sum\limits_{n=N}^{M-1} 2^{n(p\beta+p-1)} \om_{2^{n+1}}^{p\alpha}
\\& \asymp \sum\limits_{k=2^N+1}^{2^M}(k+1)^{p\beta+p-2}\om_{2k}^{p\alpha}  = \|f_{2^N+1,2^M}\|_{HL(p)}^p,
\end{align*}
A similar calculation shows that 
\begin{equation}\label{eq:mpr2}
\|f_{2^{N+1}+1, 2^M}\|_{D^p_{p-1}} \asymp \|f_{2^{N}+1, 2^{M+1}}\|_{D^p_{p-1}},\quad M>N+1.
\end{equation}
Next, if $N> 2$, there is 
 $N^\star, M^\star\in \N$ such that $2^{N^\star}\leq N-1<2^{N^\star+1}$ and $2^{M^\star}\leq M-1<2^{M^\star+1}$,
 so $N^\star +1<M^\star$. 
  Then, by 
  \cite[Theorem~2.1(b)]{MatPav} and the boundedness of the Riesz projection,
  \eqref{eq:mpr1} and \eqref{eq:mpr2}
\begin{align*}
\|f_{N, M}\|_{D^p_{p-1}}^p 
 &
\asymp  2^{-pN^\star} \left\|\sum\limits_{k=N-1}^{2^{N^\star+1}-1} (k+1)\fg_{N,M}(k+1) z^k\right\|_{H^p}^p +
 \| f_{2^{{N^\star}+1}+1, 2^{M^\star}}\|^p_{D^p_{p-1}}
\\ &+ 
 2^{-pM^\star} \left\|\sum\limits_{k=2^{M^\star}}^{M-1} (k+1)\fg_{N,M}(k+1) z^k\right\|_{H^p}^p 
 \\
&\lesssim \|f_{2^{N^\star}+1, 2^{M^\star+1}}\|_{D^p_{p-1}}^p \asymp  \|f_{2^{N^\star+1}+1, 2^{M^\star}}\|_{D^p_{p-1}}^p  \asymp \|f_{2^{N^\star+1}+1, 2^{M^\star} }\|_{HL(p)}^p\lesssim \|f_{N,M}\|_{HL(p)}^p.
\end{align*}
On the other hand, 
$$\|f_{N, M}\|_{D^p_{p-1}}^p \gtrsim \|f_{2^{N^\star+1}+1, 2^{M^\star}}\|_{D^p_{p-1}}^p \asymp \|f_{2^{N^\star}+1, 2^{M^\star+1}}\|_{D^p_{p-1}}^p
 \asymp \|f_{2^{N^\star}+1, 2^{M^\star+1}}\|_{HL(p)}^p \geq \|f_{N,M}\|_{HL(p)}^p.$$
Then, bearing in mind \eqref{eq:HpDpp<2} and \eqref{eq:HpDpp>2}, we obtain $\|f_{N, M}\|_{HL(p)}\asymp  \|f_{N, M}\|_{H^p}\asymp  \|f_{ N, M}\|_{D^p_{p-1}}$ for each $N> 2$.

If $N\in\{0,1,2\}$, the previous argument together with minor modifications
implies \eqref{eq:mpr0}. This finishes the proof.
\end{proof}

Now we are ready to prove the necessity part of Theorem~\ref{bounded Xp}.
\begin{proposition}
\label{cond nec Xp}
Let $\om$ be a radial weight and $1<p<\infty$. If $X_p, Y_p\in\{H(\infty,p),H^p, D^p_{p-1}, HL(p)\}$, $T\in \{\Ho, \widetilde{\Ho}\}$, and  
$T : X_p\to Y_p$ is a bounded operator. Then, $\omega\in\DDD$ and 
\begin{equation}
\label{Mp}
M_p(\om)=\sup\limits_{N\in \N} \left(\sum\limits_{n=0}^N \frac{1}{(n+1)^2 \om_{2n+1}^p}\right)^{\frac{1}{p}}\left(\sum\limits_{n=N}^{\infty} \om_{2n+1}^{p'}(n
+1)^{p'-2}\right)^{\frac{1}{p'}}<\infty.
\end{equation}

\end{proposition}

\begin{proof}
In order to obtain both conditions, $\omega\in\DDD$ and $M_p(\om)<\infty $, we are going to work with
families of test functions constructed in Lemma~\ref{lema f test Mp}. Since they have  
non-negative  Maclaurin coefficients, it is enough to prove the result for $T= H_{\om }$.
Take $f \in \H(\D)$ such that $\fg(n) \geq 0$ for all $n \in \N$.
\par {\bf {First Step.} } 
We will prove that $\omega\in\DD$. By Lemma~\ref{HLP C LP}, it is enough to deal with the case $Y_p=H(\infty,p)$.

Observe that
$M_\infty(r,H_\omega(f))=\sum_{n=0}^\infty\frac{1}{2(n+1)\omega_{2n+1}}\left(\sum_{k=0}^\infty \widehat{f}(k)\omega_{n+k}\right)r^n$. 
Now, consider the test functions $f_N(z)=\sum\limits_{n=0}^N \frac{1}{(n+1)^{1-\frac{1}{p-1}}} z^n, N\in \N$. Given that
$\sum\limits_{k=0}^N \frac{1}{(k+1)^{1-\frac{1}{p-1}}}\asymp (N+1)^{\frac{1}{p-1}}$, 
\begin{equation*}\begin{split}
M_\infty(r,H_\omega(f_N))&\ge\sum_{n=6N}^{7N}\frac{1}{2(n+1)\omega_{2n+1}}\left(\sum_{k=0}^N \widehat{f_N}(k)\omega_{n+k}\right)r^n
\\ & \gtrsim \sum_{n=6N}^{7N}\frac{\omega_{n+N}}{(n+1)\omega_{2n+1}}\left(\sum\limits_{k=0}^N \frac{1}{(k+1)^{1-\frac{1}{p-1}}}\right)r^n
\\ & \gtrsim (N+1)^{\frac{1}{p-1}} \frac{\omega_{8N}}{\omega_{12N}}r^{7N}
, \quad  N\in \N,\quad 0\le r<1.
\end{split}\end{equation*}
So, $$\| H_\omega(f_N)\|^p_{H(\infty,p)}\gtrsim (N+1)^{\frac{1}{p-1}}\left( \frac{\omega_{8N}}{\omega_{12N}}\right)^p,\quad  N\in \N.$$
By Lemma~\ref{lema f test Mp} and Lemma~\ref{HLP C LP},
$$\| f_N\|^p_{X_p}\lesssim\| f_N\|^p_{HL(p)}= \sum\limits_{n=0}^N \frac{1}{(n+1)^{1-\frac{1}{p-1}}}\asymp (N+1)^{\frac{1}{p-1}}.$$
Consequently,
\begin{equation*}\begin{split}
(N+1)^{\frac{1}{p-1}}\left( \frac{\omega_{8N}}{\omega_{12N}}\right)^p & \lesssim 
\| H_\omega(f_N)\|^p_{H(\infty,p)} \lesssim  \|f_N\|^p_{X_p}\lesssim
(N+1)^{\frac{1}{p-1}},\quad  N\in \N.
\end{split}\end{equation*}

Therefore, there is $C=C(\omega,p)$ such that 
$\om_{8N}\leq C \om_{12N},\quad N \in \N$. 
From now on, for each $x\in\mathbb{R}$, $E[x]$ denotes the biggest integer\,$\le x$.
For any $x\ge 120$,   take $N\in \N$ such that $8N\le x<8N+8$, and then
$$\om_x\le \om_{8N}\leq C \om_{12N}\le C \om_{8E\left[\frac{3N}{2} \right]}
\le  C^2   \om_{12E\left[\frac{3N}{2} \right]}
\le C^2 \om_{18N-12}\le C^2 \om_{16N+16}\le C^2 \om_{2x}
.$$
So, $\omega\in \DD$ by Lemma~\ref{caract. pesos doblantes}.

\par {\bf {Second Step.} } 
We will prove that $M_p(\om) < \infty$.

{\bf{Case $\mathbf{Y_p=HL (p)}$}}.   Set an arbitrary $N \in \N$. Then, bearing in mind  that $\{\omega_k\}_{k=0}^\infty$ is decreasing,
\begin{equation}\begin{split}\label{eq:testing}
\left( \sum\limits_{n=0}^{N}\frac{1}{(n+1)^2\om_{2n+1}^p}\right) \left(\sum\limits_{k=N}^{\infty}\fg(k)\om_{2k+1}\right)^p 
&
\le
\sum\limits_{n=0}^{\infty}\frac{1}{(n+1)^2\om_{2n+1}^p} \left(\sum\limits_{k=0}^{\infty}\fg(k)\om_{n+k}\right)^p
\\ & \lesssim 
\|H_{\om}\|_{X_p\to HL(p) }^p \|f\|^p_{X_p}.
\end{split}\end{equation}

Take $M, N \in \N$, $M>4N+1$, and consider the family of test polynomials 
\begin{equation}\label{eq:testw}
f_{N,M}(z)=\sum\limits_{k=N}^{M}\om_{2k+1}^{p'-1}(k+1)^{p'-2}z^k,\, z \in \D.
\end{equation} 
Then, Lemmas~\ref{HLP C LP} and \ref{lema f test Mp} yield $$\sum\limits_{k=N}^{M}\om_{2k+1}^{p'}(k+1)^{p'-2}= \|f_{ N, M}\|^p_{HL(p)}\gtrsim  \|f_{N,M}\|^p_{X_p}$$
where the constants do not depend on $M$ or $N$.

So, testing this family of functions in \eqref{eq:testing}, there exists $C=C(p, \om)>0$ such that
\begin{equation*}
\left(\sum\limits_{n=0}^{N}\frac{1}{(n+1)^2\om_{2n+1}^p}\right) \left(\sum\limits_{k=N}^{M}\om_{2k+1}^{p'}(k+1)^{p'-2}\right)^{p-1} \leq C,
\quad \text{for any $M, N \in \N$, $M>4N+1$}.
\end{equation*} 
By letting $M\to\infty$, and taking the supremum in $N\in\N$, \eqref{Mp} holds.

{\bf{Case $\mathbf{Y_p\in \{H(\infty,p),H^p,D^p_{p-1}\}}$}}.
Let  $f_{ N,M}$ be the functions defined in \eqref{eq:testw}, then $\Ho (f_{N,M})=\widetilde{\Ho}(f_{N,M})$.  
This together with the fact that $\om \in \DD$ and Lemma~\ref{le:Htildeequiv}, yields
 $$ \|\Ho (f_{N,M})\|_{Y_p}\asymp \|\Ho (f_{N,M})\|_{HL(p)} $$
  where the costants in the inequalities do not depend on $M$ or $N$.

Therefore, using  Lemma~\ref{HLP C LP}, Lemma~\ref{le:Htildeequiv} and Lemma \ref{lema f test Mp},
there exists $C=C(p, \om)>0$ such that
$$\|\Ho (f_{N,M})\|_{HL(p)} \leq C \|f_{ N,M}\|_{HL(p)}.$$
So, arguing as in the  case $Y_p=HL(p)$, we obtain $M_p(\om)<\infty.$

\par {\bf {Third Step.} } 
We will prove that the condition $M_p(\om)<\infty$ implies that  $\om \in \M$.
Indeed, set $K,M \in \N$, $K,M >1$ and $N \in \N$. By \eqref{Mp},
\begin{align*}
\infty &> M_p(\om) \geq \left(\sum\limits_{j=N}^{KN} \frac{1}{(j+1)^2 \om_{2j+1}^p}\right)^{\frac{1}{p}}\left(\sum\limits_{j=KN}^{(K+M)N-1} \om_{2j+1}^{p'}(j+1)^{p'-2}\right)^{\frac{1}{p'}}
\\&\geq \frac{\om_{2(K+M)N}}{\om_{2N}} \left(\sum\limits_{j=N}^{KN} \frac{1}{(j+1)^2 }\right)^{\frac{1}{p}}\left(\sum\limits_{j=KN}^{(K+M)N-1}(j+1)^{p'-2}\right)^{\frac{1}{p'}},
\end{align*} 
 So, there is $C=C(p)>0$ such that
 
\begin{equation}
\label{c:naturales} \om_{2N}\geq \om_{2(K+M)N}\frac{1}{M_p(\om)}C\left((K+M)^{p'-1}-K^{p'-1}\right)^{\frac{1}{p'}},
\text{ for all } N\in \N.
\end{equation}
Now, fix $K>1$ and take  $M\in \N$ large enough  such that 
$$\frac{1}{M_p(\om)}C\left((K+M)^{p'-1}-K^{p'-1}\right)^{\frac{1}{p'}}=C(K,M,p, \om)>1.$$
Let $x\geq 1$ and take $N \in \N$ such that $2N-2\leq x <2N$. Then,
 by \eqref{c:naturales}
$$ \om_x \geq \om_{2N}\geq C(K,M,p, \om)
\om_{2(K+M)N}\geq C(K,M,p, \om) \om_{(K+M)x+2(K+M)}\ge
C(K,M,p, \om) \om_{3(K+M)x},$$
so  $\om \in \M$.
 Since $\om \in \DD$,  \cite[Theorem~3]{PR19} yields $\om \in \DDD$. The proof is finished.

\end{proof}

\subsection{Sufficiency part of Theorem~\ref{bounded Xp}.}

For the purpose of proving Theorem~\ref{bounded Xp} we need some additional preparations.
In particular,  we aim for reformulating 
the necessary discrete condition on the moments of the radial weight $\omega$, 
$M_p(\omega)<\infty$,
as a continuous inequality in terms of $\omg(r)$. Observe that a radial weight $\omega$  satisfies the condition 
$$K_{p,c}(\om)= \sup\limits_{0<r<1} \left(1+\int_0^r \frac{1}{\omg(t)^p} dt\right)^{\frac{1}{p}}
\left(\int_r^1 \left(\frac{\omg(t)}{1-t}\right)^{p'}\,dt\right)^{\frac{1}{p'}}<\infty$$ if and only if $M_{p,c}(\om)<\infty.$
This fact will be used repeatedly throughout the paper.

\begin{lemma}\label{reformulacionMp}
Let $1<p<\infty$ and $\om \in \DD$. Set
$$K_{p,c}(\om)= \sup\limits_{0<r<1} \left(1+\int_0^r \frac{1}{\omg(t)^p} dt\right)^{\frac{1}{p}}
\left(\int_r^1 \left(\frac{\omg(t)}{1-t}\right)^{p'}\,dt\right)^{\frac{1}{p'}}.$$
Then,
$$\int_0^1 \left(\frac{\omg(t)}{1-t}\right)^{p'}\,dt  \asymp \sum\limits_{k=0}^{\infty} \om_{2k+1}^{p'}(k+1)^{p'-2}$$
and
$$M_p(\om)\asymp K_{p,c}(\omega)
.$$
\end{lemma}
\begin{proof}
 
Let $0<r<1$ and set $N \in \N$ such that $1-\frac{1}{N}\leq r < 1-\frac{1}{N+1}$. Then, by using Lemma \ref{caract. pesos doblantes},
\begin{align*}
 \sum\limits_{k=0}^N \frac{1}{(k+1)^2 \om_{2k+1}^p} &\asymp \sum\limits_{k=0}^N \frac{1}{(k+1)^2 \omg \left(1-\frac{1}{k+1}\right)^p} \gtrsim 1+ \sum\limits_{k=1}^N \int_k^{k+1} \frac{1}{x^2  \omg \left(1-\frac{1}{x}\right)^p} dx
 \\& = 1+ \int_0^{1-\frac{1}{N+1}} \frac{1}{\omg(s)^p} ds \geq 1+ \int_0^{r} \frac{1}{\omg(s)^p} ds.
\end{align*}
In addition, by Lemma \ref{caract. pesos doblantes}  again,
\begin{align}\label{eq:mpr6}
\int_r^1 \left(\frac{\omg(t)}{1-t}\right)^{p'}\,dt &\leq \int_{1-\frac{1}{N}}^1 \left(\frac{\omg(t)}{1-t}\right)^{p'}\,dt= \sum\limits_{k=N}^{\infty} \int_{1-\frac{1}{k}}^{1-\frac{1}{k+1}}  \left(\frac{\omg(t)}{1-t}\right)^{p'}\,dt \nonumber
\\&\leq \sum\limits_{k=N}^{\infty} \omg \left(1-\frac{1}{k}\right)^{p'}\int_{1-\frac{1}{k}}^{1-\frac{1}{k+1}} \frac{1}{(1-t)^{p'}}\,dt \lesssim \sum\limits_{k=N}^{\infty} \om_{2k+1}^{p'}(k+1)^{p'-2}.
\end{align}
Therefore, $K_{p,c}(\omega)\lesssim M_p(\om)$.

Conversely, in order to obtain the reverse inequality, 
a similar argument to \eqref{eq:mpr6} yields
\begin{align*}
 \sum\limits_{k=0}^N \frac{1}{(k+1)^2 \om_{2k+1}^p} &\asymp 1+ \sum\limits_{k=1}^N \frac{1}{(k+1)^2 \omg \left(1-\frac{1}{k}\right)^p} \lesssim 1+ \int_0^{1-\frac{1}{N+1}} \frac{1}{\omg(s)^p} ds.
\end{align*}
Now, on the one hand, if $r\leq \frac{1}{2}$ then
 $\sum\limits_{k=0}^N \frac{1}{(k+1)^2 \om_{2k+1}^p} \lesssim 1\le  1+ \int_0^{r} \frac{1}{\omg(s)^p} ds.$ On the other hand, if $\frac{1}{2}\leq r <1$,
$$ \sum\limits_{k=0}^N \frac{1}{(k+1)^2 \om_{2k+1}^p} \lesssim 1+ \int_0^{r} \frac{1}{\omg(s)^p} ds + 
\int_{1-\frac{1}{N}}^{1-\frac{1}{N+1}} \frac{1}{\omg(s)^p}ds\lesssim 1+ \int_0^{r} \frac{1}{\omg(s)^p} ds +\frac{1}{N    \omg\left(1-\frac{1}{N+1}\right)^{p}}$$
So, Lemma \ref{caract. pesos doblantes}  yields
\begin{align*}
 \sum\limits_{k=0}^N \frac{1}{(k+1)^2 \om_{2k+1}^p} &\lesssim 1+ \int_0^{r} \frac{1}{\omg(s)^p} ds +\frac{1-r}{ \omg\left(2r-1 \right)^{p}}
 \\& \asymp 1+ \int_0^{r} \frac{1}{\omg(s)^p} ds + \int_{2r-1}^{r} \frac{1}{\omg(s)^p} ds 
 \\ &\lesssim 1+ \int_0^{r} \frac{1}{\omg(s)^p} ds,\quad
 \frac{1}{2}\leq r <1.
\end{align*}

Next, 
\begin{align}\label{eq:mpr7}
\sum\limits_{k=N}^{\infty} \om_{2k+1}^{p'}(k+1)^{p'-2} &\asymp \sum\limits_{k=N}^{\infty}   \omg\left(1-\frac{1}{k+2}\right)^{p'}(k+1)^{p'} \int_{1-\frac{1}{k+1}}^{1-\frac{1}{k+2}}dt \nonumber
\\ & \lesssim \int_{1-\frac{1}{N+1}}^1 \left(\frac{\omg(t)}{1-t}\right)^{p'}\,dt  \leq \int_{r}^1 \left(\frac{\omg(t)}{1-t}\right)^{p'}\,dt, 
\end{align}
and consequently,
$ M_p(\om)\lesssim K_{p,c}(\omega)$.
Finally, \eqref{eq:mpr6}  and \eqref{eq:mpr7} imply
$$\int_0^1 \left(\frac{\omg(t)}{1-t}\right)^{p'}\,dt  \asymp \sum\limits_{k=0}^{\infty} \om_{2k+1}^{p'}(k+1)^{p'-2}.$$
This finishes the proof.
\end{proof}

We will also need the following description of the class $\DDD$. 
\begin{lemma}\label{le:descriptionDDD}
Let $\omega$ be a radial weight. Then the following conditions are equivalent:
\begin{enumerate}
\item[(i)] $\omega\in\DDD$;
\item[(ii)] The function defined as $ \widetilde{\omega}(r)=\frac{\widehat{\omega}(r)}{1-r}$, $0\leq r<1$,
is a radial weight and satisfies 
$$\widehat{\omega}(r)\asymp \widehat{\widetilde{\omega}}(r),\quad 0\le r<1.$$  
\end{enumerate}
\end{lemma}
\begin{proof}
(i)$\Rightarrow$(ii). By Lemma~\ref{caract. D check}, there is $\alpha>0$ such that
$$\int_{r}^1 \widetilde{\omega}(s)\,ds\lesssim \frac{\widehat{\omega}(r)}{(1-r)^\alpha}\int_r^1 (1-s)^{\alpha-1}\,ds\lesssim \widehat{\omega}(r),\quad 
0\le r<1,$$
which, in particular, implies that $\widetilde{\omega}$ is a radial weight.
On the other hand, by Lemma~\ref{caract. pesos doblantes}, there is $\beta>0$ such that 
$$\int_{r}^1 \widetilde{\omega}(s)\,ds 
\gtrsim  \frac{\widehat{\omega}(r)}{(1-r)^\beta} \int_{r}^1 (1-s)^{\beta-1}\,ds\gtrsim \widehat{\omega}(r),\quad 
0\le r<1.$$

Reciprocally, if (ii) holds, there are $C_1,C_2>0$ such that
$$C_1\widehat{\omega}(r)\le \widehat{\widetilde{\omega}}(r) \le C_2\widehat{\omega}(r),\quad 0\le r<1.$$
So, for any $K>1$,
$$\widehat{\omega}(r)\ge\frac{1}{C_2} \int_{r}^{1-\frac{1-r}{K}}\widetilde{\omega}(s)\,ds\ge 
\frac{\log K}{C_2} \widehat{\omega}\left(1-\frac{1-r}{K}\right),\quad 0\le r<1.$$
Therefore, taking $K$ such that $\frac{\log K}{C_2}>1$, $\omega\in \Dd$.

Moreover, for any $K>1$
\begin{equation*}\begin{split}
\widehat{\omega}(r) \le \frac{1}{C_1}\widehat{\widetilde{\omega}}(r)
 \le \frac{\log K}{C_1}\widehat{\omega}(r)+\frac{1}{C_1}\int_{1-\frac{1-r}{K}}^1\widetilde{\omega}(s)\,ds\quad 0\le r<1.
\end{split}\end{equation*}
If $\frac{\log K}{C_1}<1$ , then
$$\widehat{\omega}(r)\le \frac{C_2}{1-\frac{\log K}{C_1}}\widehat{\omega}\left(1-\frac{1-r}{K}\right),\quad 0\le r<1.$$
So,  $\omega\in\DD$. This finishes the proof.

\end{proof}

The previous lemma  may be used to prove  that a  differentiable non-decreasing  function $h:[0,1)\to [0,\infty)$ belongs to $L^p_{\om , [0,1)}$ if and only if it belongs to  $L^p_{\widetilde{\omega}, [0,1)}$. 
This result is essential for our purposes.
 In particular, bearing in mind Lemma~\ref{le:descriptionDDD} and two integration by parts, 
\begin{equation}\label{eq:intparts}
 \int_0^1 h(t)\om(t)\,dt\lesssim 
h(0)\widehat{\omega}(0)
+\int_0^1 h(t)\widetilde{\omega}(t)\,dt,
\end{equation}
for any differentiable non-decreasing function $h:[0,1)\to [0,\infty)$ .

Bearing in mind Lemma \ref{reformulacionMp}, our next result ensures that the Hilbert-type operators $H_\omega$ and $\widetilde{H_\omega}$ are well defined 
on $X_p\in \{H(\infty,p),H^p, D^p_{p-1}, HL(p)\}$, $1<p<\infty$ when $\om\in\DDD$ and $M_p(\omega)<\infty$.
\begin{lemma}\label{gooddefinitionv2}
Let $\om\in\DDD$  and $1<p<\infty$ such that $\int_0^1 \left(\frac{\omg(t)}{1-t}\right)^{p'}\,dt<\infty$.

Then 
$$ \int_0^1 M_\infty(t,f)\om(t)\,dt\lesssim 
\|f\|_{H(\infty,p)}\left(\int_0^1 \left(\frac{\omg(t)}{1-t}\right)^{p'}\,dt \right)^{1/p'},\quad f\in \H(\D).$$
In particular,  
$T(f)\in \H(\D)$ for any $f\in X_p$, where $X_p\in \{H(\infty,p),H^p, D^p_{p-1}, HL(p)\}$ and $T\in\{H_\omega, \widetilde{H_\omega}\}$.
\end{lemma}
\begin{proof}

By \eqref{eq:intparts}
\begin{equation}\label{eq:inomegatilde}
\int_0^1 M_\infty(t,f)\om(t)\,dt 
 \le |f(0)|\widehat{\omega}(0)+\int_0^1 M_\infty(t,f)\widetilde{\om}(t)\,dt.
 \end{equation}
Then,  by  H\"older's inequality
  \begin{align*}
\int_0^1 M_\infty(t,f)\om(t)\,dt &
  \le  |f(0)|\widehat{\omega}(0)+ \left( \int_0^1 M^p_\infty(t,f)\,dt\right)^{1/p}\left(\int_0^1 \widetilde{\om}(t)^{p'}\,dt \right)^{1/p'}
\\ & \lesssim 
\|f\|_{H(\infty,p)}\left(\int_0^1 \widetilde{\om}(t)^{p'}\,dt \right)^{1/p'}<\infty, \quad f\in H(\infty,p).
\end{align*}
Joining the above chain of inequalities  with Lemma~\ref{HLP C LP}, the proof is finished.
\end{proof}

Next,
for $p,q>0$ and $\alpha >-1$, let
$H^1(p,q, \alpha)$ denote the space of $f \in \H(\D)$ such that
$$ \nm{f}_{H^1(p,q, \alpha)}=\left(\abs{f(0)}^p +\int_0^1 M_q^p(r,f')(1-r)^{\a}dr\right)^{\frac{1}{p}}<\infty.$$
It is worth mentioning that $H^1\left (p,p, p-1\right)=D^p_{p-1}$.

The following  inequality will be used in the proof of Theorem~\ref{bounded Xp}. It
was proved in \cite[Corollary~3.1]{MatPav}.

\begin{letterlemma}\label{le:Hpmixto}
Let $1<q<p<\infty$. Then, 
$$\|f\|_{H^p}\lesssim  \nm{f}_{H^1\left (p,q, p\left(1-\frac{1}{q}\right)\right)}, \quad f\in \H(\D).$$
\end{letterlemma}

Now, we are ready to prove the main result of this section.

\vspace{1em}
\begin{Prf} {\em{Theorem~\ref{bounded Xp}. }}
The implication (i)$\Rightarrow$(ii)
was proved in  Proposition~\ref{cond nec Xp}. The implication (ii)$\Rightarrow$(iii) is clear, 
and  (iii)$\Rightarrow$(ii) follows from the third step in the proof of Proposition~\ref{cond nec Xp}.
On the other hand, bearing in mind that $M_{p,c}(\om)<\infty$ if and only if $K_{p,c}(\om)<\infty$, (iii)$\Leftrightarrow$(iv) follows from Lemma~\ref{reformulacionMp}.
 Then, it is enough to prove (ii)$\Rightarrow$(i).
 
\par{\bf{(ii)$\mathbf{\Rightarrow}$(i)}}.

\vspace{1em}

{\bf { First Step.}}
We will prove the inequality
\begin{equation}\label{eq:jp1}
\| T(f)\|_{Y_p}\lesssim \|f\|_{X_p}+ \| \widetilde{H_\omega}(f)\|_{Y_p},\quad f\in X_p.
\end{equation}

By Lemma~\ref{HLP C LP} and Lemma~\ref{le:Hpmixto}, it is enough to prove 
\begin{equation}\label{eq:jp2}
\| H_\omega(f)\|_{H^1\left (p,q, p\left(1-\frac{1}{q}\right)\right)}\lesssim \|f\|_{X_p}+ \| \widetilde{H_\omega}(f)\|_{Y_p},\quad 
1<q,p<\infty,\quad f\in X_p.
\end{equation}
Let $f \in X_p$. Then, Lemma \ref{reformulacionMp} and Lemma~\ref{gooddefinitionv2} ensure that $H_\om(f)\in\H(\D)$. By
 \cite[Theorem~2.1]{MatPav} 
 \begin{align}\label{eq:mpr8}
 \nm{H_{\omega} (f)}_{H^1\left (p,q, p\left(1-\frac{1}{q}\right)\right)}^p \asymp |\Ho(f)(0)|^p+ |\Ho(f)'(0)|^p + \sum\limits_{n=0}^{\infty} 
 2^{-n\left(p\left(1-\frac{1}{q}\right)+1\right)}\nm{\De_n (H_{\omega}(f))'}_{H^q}^p.
\end{align}
Due to $$ (H_{\omega} (f))'(z)=\sum\limits_{n=0}^{\infty}\frac{n+1}{2(n+2)\om_{2n+3}}\left(\int_0^1 f(t) t^{n+1}\om(t)dt\right)z^n,$$
and  using the proof of \cite[Lemma~7]{GaGiPeSis}, Lemma \ref{lema de los lambda} and \eqref{eq:Dn}, 
$$ \|\De_n (H_{\omega}(f))'\|_{H^q}^p\lesssim \frac{\left(\int_0^1 t^{2^{n-2}+1} |f(t)| \om (t)dt\right)^p}{\om_{2^{n+2}+3}^p}2^{np\left(1-\frac{1}{q}\right)}, \quad  n \geq 3.$$
Hence, by using Lemma \ref{caract. pesos doblantes},
\begin{align}\label{eq:mpr9}
\sum\limits_{n=3}^{\infty} 2^{-n\left(p\left(1-\frac{1}{q}\right)+1\right)}\nm{\De_n (H_{\omega}(f))'}_{H^q}^p \lesssim \sum\limits_{n=3}^{\infty}  \frac{2^{-2n}}{\om_{2^{n+2}}^p} \sum\limits_{k=2^{n-3}}^{2^{n-2}} \left(\int_0^1 t^{k} |f(t)| \om (t)dt\right)^p
\\
\lesssim \sum\limits_{k=1}^{\infty} \frac{\left(\int_0^1 t^{k} |f(t)| \om (t)dt\right)^p}{\om_{2k+1}^p (k+1)^2}\lesssim \| \widetilde{\Ho} (f)\|_{HL(p)}^p
\nonumber.
\end{align}

In addition, by Lemma~\ref{HLP C LP} and Lemma~\ref{gooddefinitionv2}
\begin{align}\label{eq:mpr10}
|\Ho(f)(0)|^p +|\Ho(f)'(0)|^p +\sum \limits_{n=0}^2 2^{-np}\nm{\De_n (H_{\omega}(f))'}_{H^p}^p \lesssim \|f\|_{X_p}^p.
\end{align}

Therefore, by putting together \eqref{eq:mpr8}, \eqref{eq:mpr9} and \eqref{eq:mpr10}

\begin{align*}
 \nm{H_{\omega} (f)}_{H^1\left(p,q, p\left(1-\frac{1}{q}\right)\right)}^p 
\lesssim  \|f\|_{X_p}^p +  \| \widetilde{\Ho} (f)\|_{HL(p)}^p.
\end{align*}
The above inequality, together with Lemma~\ref{le:Htildeequiv}, yields \eqref{eq:jp2}.

{\bf{Second Step.}} We will prove the inequality
\begin{equation}\label{eq:jp11}
\| \widetilde{H_\omega}(f)\|_{D^p_{p-1}}\lesssim \|f\|_{H(\infty,p)},\quad f\in H(\infty,p).
\end{equation}

We denote by 
\begin{equation}\label{eq:Gwt}
G^\omega_t(z)=\frac{d}{dz}\left(\frac{1}{z}\int_0^z B^{\om}_t(\z)d\z \right).
\end{equation} 
By \cite[Lemma~B]{PelRosa21}
$$ M_p (r, G_t^{\om})\lesssim\left( 1+\int_0^{rt}\frac{ds}{\omg(s)(1-s)^p} \right)^{1/p}
\lesssim \frac{1}{\omg(rt) (1-rt)^{1-\frac{1}{p}}},\quad 0\le r,t<1,$$
which together with
 Minkowski's inequality yields
\begin{align*}
 \nm{\widetilde{H_{\omega}}(f)}_{D^p_{p-1}}^p
    &\lesssim \abs{\widetilde{H_{\omega}}(f)(0)}^p +\int_0^1 \left(\int_0^1 \vert f(t)\vert \om(t) M_p (r, G_t^{\om} ) dt\right)^p (1-r)^{p-1}\,dr
   \\ &\lesssim \abs{\widetilde{H_{\omega}}(f)(0)}^p +\int_0^1 \left(\int_0^1 \frac{\vert f(t)\vert \om(t)}{\omg(rt) (1-rt)^{1-\frac{1}{p}}} dt\right)^p (1-r)^{p-1}\,dr
   \\ &\leq \abs{\widetilde{H_{\omega}}(f)(0)}^p +\int_0^1 \left(\int_0^1 \frac{ M_{\infty}(t,f)  \om(t)}{\omg(rt) (1-rt)^{1-\frac{1}{p}}} dt\right)^p (1-r)^{p-1}\,dr.
\end{align*}
Now, 
by \eqref{eq:intparts}
\begin{equation}\begin{split}\label{eq:jp4}
 \nm{\widetilde{H_{\omega}}(f)}_{D^p_{p-1}}^p
    &\lesssim \abs{\widetilde{H_{\omega}}(f)(0)}^p + |f(0)|^p +\int_0^1 \left( \int_0^1 \frac{ M_{\infty}(t,f) }{\omg(rt) (1-rt)^{1-\frac{1}{p}}} \frac{\omg(t)}{1-t} dt\right)^p (1-r)^{p-1} \,dr.
 \end{split}\end{equation}
 Next, by Lemma \ref{reformulacionMp}, $M_{p,c}(\om)<\infty$ holds, so \cite[Theorem 2]{Muckenhoupt1972} yields
\begin{equation}\begin{split}\label{eq:jp5}
\int_0^1 \left( \int_r^1 \frac{ M_{\infty}(t,f) }{\omg(rt) (1-rt)^{1-\frac{1}{p}}} \frac{\omg(t)}{1-t} dt\right)^p (1-r)^{p-1} \,dr & \asymp \int_0^1 \left( \int_r^1 M_{\infty}(t,f) \frac{\omg(t)}{1-t} dt\right)^p \frac{1}{\omg(r)^{p}} \,dr  
\\& \lesssim \|f\|^p_{H(\infty,p)}
 \end{split}\end{equation} 

On the other hand, by  \cite[Theorem 1]{Muckenhoupt1972}, 
\begin{equation}\begin{split}\label{eq:jp6}
 \int_0^1 \left( \int_0^r \frac{ M_{\infty}(t,f) }{\omg(rt) (1-rt)^{1-\frac{1}{p}}} \frac{\omg(t)}{1-t} dt\right)^p (1-r)^{p-1} \,dr 
 &\asymp \int_0^1 \left( \int_0^r \frac{ M_{\infty}(t,f) }{ (1-t)^{2-\frac{1}{p}}} dt\right)^p (1-r)^{p-1} \,dr  \\ &\lesssim \|f\|^p_{H(\infty,p)},
 \end{split}\end{equation}
where in the last inequality we have used that
$\sup\limits_{0<r<1} \left(\int_r^1(1-t)^{p-1}\right)^{\frac{1}{p}} \left(\int_0^r(1-t)^{-1-p'}\right)^{\frac{1}{p'}}<\infty$.
So, joining Lemma~\ref{gooddefinitionv2}, \eqref{eq:jp4}, \eqref{eq:jp5} and \eqref{eq:jp6},  we get \eqref{eq:jp11}.

{\bf{Third Step.}} Since $\omega\in\DD$, by Lemma~\ref{le:Htildeequiv} 
$$\| \widetilde{H_\omega}(f)\|_{D^p_{p-1}}\asymp \| \widetilde{H_\omega}(f)\|_{Y_p}, \quad f\in X_p,$$
for $Y_p\in\{H(\infty,p),H^p, D^p_{p-1}, HL(p)\}$. This, together with \eqref{eq:jp1}, \eqref{eq:jp11} and Lemma~\ref{HLP C LP} implies
\begin{equation*}\begin{split}
\| T(f)\|_{Y_p} &\lesssim \|f\|_{X_p}+ \| \widetilde{H_\omega}(f)\|_{Y_p}
\\ & \asymp  \|f\|_{X_p}+ \| \widetilde{H_\omega}(f)\|_{D^p_{p-1}}
\\ &\lesssim \|f\|_{X_p}+\|f\|_{H(\infty,p)}
\\ & \lesssim \|f\|_{X_p}, \quad f\in X_p.
\end{split}\end{equation*}

This finishes the proof.
\end{Prf}

\subsection{ $\mathbf{H_\omega: X_p\to Y_p}$ versus
  $\mathbf{H_\omega: L^p_{[0,1)}\to Y_p,\, 1<p<\infty}$.}

\par An additional byproduct of Theorem~\ref{bounded Xp} is the following improvement of \cite[Theorem~3]{PelRosa21}.

\vspace{1em}
\begin{Prf}{\em{Corollary~\ref{th:lphp}.}}
(i)$\Rightarrow$(ii). By Lemma~\ref{HLP C LP}, $T:Y_p \to Y_p$ is bounded, and so by Theorem~\ref{bounded Xp}, $\omega\in\DDD$.
Next, by the proofs of  \cite[Theorems~3 and 4]{PelRosa21} we obtain $m_p(\om)<\infty$.

(ii)$\Rightarrow$(iii) is clear. Finally,  (iii)$\Rightarrow$(i) follows from Lemma~\ref{HLP C LP} and \cite[Theorem~3]{PelRosa21}.
\end{Prf}

Putting together  Lemma~\ref{HLP C LP}, Theorem~\ref{bounded Xp} and Corollary~\ref{th:lphp}, we deduce the following result.

\begin{corollary}\label{co:goodweights}
Let  $\om$ be a radial weight and $1<p<\infty$. Let $X_p,Y_p\in \{H(\infty,p),H^p, D^p_{p-1}, HL(p)\}$
and let $T \in \{\Ho , \widetilde{\Ho}\}$. If there exists $C>0$ such that 
\begin{equation}\label{eq:regularity}
\om(t)\leq C \frac{\omg(t)}{1-t}\text{ for all $0\leq t <1$.}
\end{equation}
Then, the following statements are equivalent:
\begin{itemize}
\item[(i)] $T:L^p_{[0,1)}\to Y_p$ is bounded;
\item[(ii)] $T: X_p \to Y_p$ is bounded;
\item[(iii)] $\om \in \DDD$ and $M_{p,c}(\om)<\infty$;
\item[(iv)] $\om \in \DD$ and $M_{p,c}(\om)<\infty$;
\item[(v)] $\om \in \DDD$ and $m_p(\om)<\infty$;
\item[(vi)] $\om \in \DD$ and $m_p(\om)<\infty$.
\end{itemize}
\end{corollary}
\begin{proof}
The implication (i)$\Rightarrow$(ii) follows from Lemma~\ref{HLP C LP}, and (ii)$\Leftrightarrow$(iii)$\Leftrightarrow$(iv) follows from Theorem \ref{bounded Xp}.
Next, (iii)$\Rightarrow$(v) is a byproduct of
the hypothesis \eqref{eq:regularity}.
The equivalence (v)$\Leftrightarrow$(vi) and the implication
 (vi)$\Rightarrow$(i) have been proved in Corollary~\ref{th:lphp}. This finishes the proof. 
\end{proof}

  Next, we will prove that there are weights $\omega\in \DDD$,   such that
   $M_{p,c}(\om)<\infty$ and $m_p(\om)=\infty$, so in particular they do not satisfy \eqref{eq:regularity}.
Consequently, the boundedness of the operator $H_\om: L^p_{[0,1)}\to Y_p$ is not equivalent to the boundedness
of the the operator $H_\om: X_p\to Y_p$, where $X_p,Y_p\in \{H(\infty,p),H^p, D^p_{p-1}, HL(p)\}$.   
    With this aim we  prove the next result, which shows that despite its innocent looking condition, the class $\DDD$ has in a sense a complex nature.

\begin{lemma}\label{le:cont}
Let  $1<p<\infty$ and $\nu\in\DDD$. Then, there exists $\omega\in\DDD$ such that $$\omg(t)\asymp \widehat{\nu}(t), \quad t\in [0,1),$$
$\omega\in L^{p'}_{[0,r_0]}$ for any $r_0\in (0,1)$ and  $\omega\notin L^{p'}_{[0,1)}$.

\end{lemma}
\begin{proof}
By Lemma~\ref{le:descriptionDDD}, $\widetilde{\nu}\in\DDD$. So, we can choose 
 $K>1$ so that $\widetilde{\nu}$ satisfies \eqref{D chek y k}. Next,
  consider the sequences 
$ r_n= 1-\frac{1}{K^n},\,t_n=r_n+a_n$, with  $$0<a_n<\min\left(r_{n+1}-r_n, \frac{\left(\widehat{\widetilde{\nu}}(r_n)\right)^p}{(n+1)^{p-1}}\right),
\quad n\in\N\cup\{0\}.$$
Let 
$$\om(t)=\sum_{n=0}^\infty h_n\chi_{[r_n,t_n]}(t)\widetilde{\nu}(t),\, t\in [0,1),\quad\text{ where }\quad h_n=\frac{\widehat{\widetilde{\nu}}(r_n)-\widehat{\widetilde{\nu}}(r_{n+1})}{\widehat{\widetilde{\nu}}(r_n)-\widehat{\widetilde{\nu}}(t_{n})}, \,n\in\N\cup\{0\}.$$
Observe that the sequence $\{h_n\}_{n=0}^\infty$ is well-defined because
$$\widehat{\widetilde{\nu}}(r_n)-\widehat{\widetilde{\nu}}(t_{n})=\int_{r_n}^{t_n}
\frac{\widehat{\nu}(s)}{1-s}\,ds
\ge 
\widehat{\nu}(t_n)\log\left( 1+\frac{a_n}{1-t_n} \right)>0,\,n\in\N\cup\{0\}.
$$
Moreover, $\om$ is non-negative and
$$\int_0^1 \om(t)\,dt=\sum_{n=0}^\infty h_n\left(\widehat{\widetilde{\nu}}(r_n)-\widehat{\widetilde{\nu}}(t_{n})\right)=\sum_{n=0}^\infty \left(\widehat{\widetilde{\nu}}(r_n)-\widehat{\widetilde{\nu}}(r_{n+1})\right)=\widehat{\widetilde{\nu}}(0)=\int_0^1 \widetilde{\nu}(t)\, dt\asymp \int_0^1 \nu(t)\,dt<\infty,$$
where in the last equivalence  we have used Lemma~\ref{le:descriptionDDD}.
\par Next, take $t\in [0,1)$ and 
 $N\in\N\cup\{ 0\}$ such that $r_N\le t<r_{N+1}$. By   Lemma~\ref{le:descriptionDDD} and Lemma~\ref{caract. pesos doblantes},
\begin{align*}
\omg(t)&\le \omg(r_N)= \sum_{n=N}^\infty \left(\widehat{\widetilde{\nu}}(r_n)-\widehat{\widetilde{\nu}}(r_{n+1})\right)=\widehat{\widetilde{\nu}}(r_{N})\asymp \widehat{\nu}(r_N)\lesssim \widehat{\nu}(t) \quad\text{and}
\\ \omg(t)&\ge \omg(r_{N+1})=\sum_{n=N+1}^\infty \left(\widehat{\widetilde{\nu}}(r_n)-\widehat{\widetilde{\nu}}(r_{n+1})\right)=\widehat{\widetilde{\nu}}(r_{N+1})\asymp \widehat{\nu}(r_{N+1})\gtrsim \widehat{\nu}(t),
\end{align*}
so $\omg(t)\asymp\widehat{\nu}(t)$ and hence $\om\in\DDD$.
\par It is clear that $\omega\in L^{p'}_{[0,r_0]}$ for any $r_0\in (0,1)$, so it only remains to prove that $\om\notin L^{p^\prime}_{[0,1)}$. 
Bearing mind \eqref{D chek y k}, we get that 
$$h_n\asymp \frac{\widehat{\widetilde{\nu}}(r_n)}{\widehat{\widetilde{\nu}}(r_n)-\widehat{\widetilde{\nu}}(t_{n})}, \quad N\in\N\cup\{ 0\}.$$
This, together with Lemma~\ref{caract. pesos doblantes} and H\"older's inequality, implies 
\begin{align*}
\int_0^1 \om(t)^{p^\prime}\,dt&=\sum_{n=0}^\infty h_n^{p^\prime} \int_{r_n}^{t_n}\left( \frac{\widehat{\nu}(t)}{1-t}\right)^{p^\prime}\, dt
\asymp \sum_{n=0}^\infty \left(\widehat{\widetilde{\nu}}(r_n)\right)^{p^\prime} \frac{\int_{r_n}^{t_n}\left( \frac{\widehat{\nu}(t)}{1-t}\right)^{p^\prime}\, dt}{\left(\int_{r_n}^{t_n} \frac{\widehat{\nu}(t)}{1-t}\, dt\right)^{p^\prime}}
\\ &
\ge 
\sum_{n=0}^\infty \left( \frac{\widehat{\widetilde{\nu}}(r_n)}{(t_n-r_n)^{1/p}}\right)^{p^\prime}
=\sum_{n=0}^\infty \left( \frac{\widehat{\widetilde{\nu}}(r_n)}{a_n^{1/p}}\right)^{p^\prime}\ge \sum_{n=0}^\infty (n+1)=\infty.
\end{align*}
\end{proof}

\begin{corollary}\label{co:comparisonLp-Xp}
Let   $1<p<\infty$ and $X_p,Y_p\in \{H(\infty,p),H^p, D^p_{p-1}, HL(p)\}$.
For each radial weigth $\nu$ such that $ Q: L^p_{[0,1)}\to Y_p$ is bounded, where $Q\in\{ H_\nu,\widetilde{H_\nu}\}$, there is a radial weight 
$\omega$ such that $$\widehat{\omega}(t)\asymp \widehat{\nu}(t),\quad t\in [0,1),$$
$\omega\in L^{p'}_{[0,r_0]}$ for any $r_0\in (0,1)$,    $T: X_p\to Y_p$ is bounded and 
 $T: L^p_{[0,1)}\to Y_p$ is not bounded. Here $T\in\{ H_\omega,\widetilde{H_\omega}\}$.
\end{corollary} 
\begin{proof}
Since $ Q: L^p_{[0,1)}\to Y_p$ is bounded, by Theorem~\ref{bounded Xp}, $\nu\in\DDD$ and $M_{p,c}(\nu)<\infty$. Now, by Lemma~\ref{le:cont} there is  
a radial weight $\omega$ such that 
$\widehat{\omega}(t)\asymp \widehat{\nu}(t)$, $\omega\in L^{p'}_{[0,r_0]}$ for any $r_0\in (0,1)$  and $\omega\not\in L^{p'}_{[0,1)}$. So, $m_p(\omega)=\infty$ and by Corollary~\ref{th:lphp},
$T: L^p_{[0,1)}\to Y_p$ is not bounded. Moreover, $\omega\in\DDD$ and $M_{p,c}(\om)<\infty$ because $\nu$ satisfies both properties, so Theorem \ref{bounded Xp} yields
$T: X_p\to Y_p$ is bounded.
\end{proof}

\subsection{Compactness of Hilbert-type operators on  $X_p$-spaces. Case \,$\mathbf{1 < p<\infty}$.}

For $X,Y$ two Banach spaces, a sublinear operator $L: X\to Y$ is said to be compact provided $L(A)$ has compact closure  for any bounded set $A\subset X$.
Once it has been understood the radial weights $\omega$ such that $H_\omega: X_p\to Y_p$ is bounded, $X_p,Y_p\in \{ H(\infty,p),D^p_{p-1}, H^p, HL(p)\}$,
$1<p<\infty$, it is natural to consider the analogous problem,  replacing boundedness by compactness. Theorem \ref{pr:nocompactxp} in this section answers this question, but firstly we need some previous results.

\begin{lemma}\label{Conv compact}
Let $1 < p< \infty$ and  $\om \in \DDD$ such that $\| \widetilde{\om}\|_{L^{p'}_{[0,1)}}<\infty$. Let
 $\{f_k\}_{k=0}^\infty\subset X_p  \in \{ H(\infty,p),D^p_{p-1}, H^p, HL(p)\}$ such that $\sup\limits_{k\in\N} \|f_k\|_{X_p}<\infty$ and $f_k\to 0$ uniformly on compact subsets of $\D$.
 Then the following statements hold:
\begin{enumerate}
\item[\rm(i)] $\int_0^1 |f_k(t)|\om(t)dt\to 0$ when $k\to\infty$.
\item[\rm(ii)] If $T\in\{H_\omega, \widetilde{H_\omega}\}$, then $T(f_k)\to 0$ uniformly on compact subsets of $\D$.
\end{enumerate}
\end{lemma}
\begin{proof}
(i). Let $\varepsilon>0$. 
By hypothesis $\int_0^1 \widetilde{\om}(t)^{p'} dt < \infty$, so there exists $0<\rho_0<1$ such that $\int_{\rho_0}^1 \widetilde{\om}(t)^{p'} dt < \varepsilon $. Moreover,  there exists $k_0$ such that for every $k\ge k_0$ and $z\in M=\overline{D(0,\rho_0)}$, $|f_k(z)|<\varepsilon$. Then, by Lemma \ref{le:descriptionDDD}, \eqref{eq:intparts}, and Hölder inequality
\begin{align*}
\int_0^1|f_k(t)|\om(t)dt& \leq |f_k(0)|\omg (0) + \int_0^1 M_{\infty}(t,f_k)\widetilde{\om}(t)dt
\\&\lesssim
 \int_0^{\rho_0} M_{\infty}(t,f_k)\widetilde{\om}(t)dt + \int_{\rho_0}^1 M_{\infty}(t,f_k)\widetilde{\om}(t)dt
\\ &<\varepsilon \int_0^{\rho_0}\widetilde{\om}(t)dt + \sup_{k\in\N} \|f_k\|_{H(\infty,p)} \int_{\rho_0}^1 \widetilde{\om}(t)^{p'} dt
\\ &<\varepsilon \left(\int_0^{1}\widetilde{\om}(t)dt + \sup_{k\in\N} \|f_k\|_{H(\infty,p)}\right)=C\varepsilon ,
\end{align*}
where in the last step we have used Lemma \ref{HLP C LP}.

(ii). Let be $M\subset\D$ a compact set and $K_t^\omega(z)=\frac1z\int_0^z B_t^\om(u)\,du$. If $z\in M$ 
\begin{align*}
|T(f_k)(z)|
\le \int_0^1 |f_k(t)|\,|K_t^\omega(z)|\om(t)dt \le \sup_{\substack{z\in M\\ t\in[0,1)}} |K_t^\omega(z)|\int_0^1 |f_k(t)|\om(t)dt
\end{align*}
Since, $M\subset\overline{D(0,\rho_0)}$, for some $\rho_0\in(0,1)$, then
\begin{align*}
\sup_{\substack{z\in M\\ t\in[0,1)}} |K_t^\omega(z)|&=\sup_{\substack{z\in M\\ t\in[0,1)}} \left|\sum_{k=0}^\infty \frac{t^kz^k}{2(k+1)\om_{2k+1}} \right|\le \sum_{k=0}^\infty \frac{{\rho_0}^k}{2(k+1)\om_{2k+1}}= C(\omega,\rho_0)<\infty,
\end{align*}
so, by (i), 
$T(f_k)\to 0$ uniformly on $M$. This finishes the proof.
\end{proof}

\begin{theorem}\label{compactXp}
Let $\omega$ be a radial weight, $1<p<\infty$, $X_p, Y_p \in \{ H(\infty,p),D^p_{p-1}, H^p, HL(p)\}$ and let $T\in\{H_\omega, \widetilde{H_\omega}\}$. Then, the following assertions are equivalent:
\begin{enumerate}
\item[\rm(i)]
  $T: X_p \to Y_p$ is compact;
  \item[\rm(ii)] For every sequence $\{f_k\}_{k=0}^\infty\subset X_p $ such that $\sup\limits_{k\in\N} \|f_k\|_{X_p}<\infty$ and $f_k\to 0$ uniformly on compact subsets of $\D$,   $\lim\limits_{k\to\infty} \|T(f_k)\|_{Y_p} =0$.
  \end{enumerate}
\end{theorem}
\begin{proof}
(i)$\Rightarrow$(ii). Let $\{f_n\}_{n=0}^\infty\subset X_p$ such that $\sup\limits_{n\in\N} \|f_n\|_{X_p}<\infty$ and $f_n\to 0$ uniformly on compact subsets of $\D$.
Assume there exist $\e>0$ and a subsequence $\{n_k\}_k \subset \N$  such that 
\begin{equation}
\label{c}
\|T(f_{n_k})\|_{Y_p} >\e, \quad\text{for any $k$}.\end{equation}
Since $T$ is compact, there exists a subsequence $\{n_{k_j}\}_j\subset \N$ and $g \in Y_p$ such that $\lim\limits_{j\to\infty} \|T(f_{n_{k_j}})-g\|_{Y_p}=0$. Moreover, Theorem \ref{bounded Xp} ensures that $\om \in \DDD$ and $M_{p,c}(\om)<\infty$, so $\| \widetilde{\om}\|_{L^{p'}_{[0,1)}}<\infty$. Therefore Lemma~\ref{Conv compact}, implies that $T(f_{n_{k_j}})\to 0$ uniformly on compact subsets of $\D$, so $\lim\limits_{j\to\infty} \|T(f_{n_{k_j}})\|_{Y_p}=0$ which yields a contradiction with \eqref{c}.

(ii)$\Rightarrow$(i). Let $\{f_n\}\subset X_p$ such that $\sup\limits_{n\in \N}\|f_n\|_{X_p}<\infty.$ 
Then, 
$\{f_n\}$ 
is uniformly bounded on compact subsets of $\D$. Then, by Montel's Theorem there exists $\{f_{n_k}\}_k$ and $f\in \H(\D)$ such that $f_{n_k}\to f$ uniformly on compact subsets of $\D$. 
Let  $g_{n_k}=f_{n_k}-f$, then $g_{n_k}\to 0$ uniformly on compact subsets of $\D$ and $\sup\limits_{k\in \N}\|g_{n_k}\|_{X_p}<\infty$.  Therefore, by hypothesis $\lim\limits_{k\to\infty} \|T(g_{n_k})\|_{Y_p} =0$, that is,   $T$ is compact.

\end{proof}

\begin{theorem}\label{pr:nocompactxp}
Let $\omega$ be a radial weight, $1 < p<\infty$,  $X_p,Y_p\in \{ H(\infty,p),D^p_{p-1}, H^p, HL(p)\}$, and let $T\in \{\Ho, \widetilde{\Ho}\}$. Then, 
$T:X_p\to Y_p$ is not compact.
\end{theorem}
\begin{proof}
Assume that $T:X_p\to Y_p$ is compact. 
For each $0<a<1$, set
\begin{align*}
f_a(z)&=\left (\frac{1-a^2}{(1-az)^2}\right )
^{1/p}=\sum_{n=0}^\infty \widehat{f_a}(n)z^n,\quad z\in
\mathbb D,
\end{align*}
where $\widehat{f_a}(n)=(1-a^2)^{1/p}\frac{\Gamma(n+2/p)}{n!\Gamma(2/p)} a^n \ge 0$. So, by Stirling's formula
\begin{equation}\label{CoeffCompact}
\widehat{f_a}(n)\asymp (1-a^2)^{1/p}(n+1)^{2/p -1}a^n,\quad n\in \N\cup\{0\}.
\end{equation}
Consequently, $\Vert f_a\Vert_{HL(p)}\asymp 1,\, a\in(0,1)$. Moreover, 
 $\Vert f_a\Vert_{H(\infty,p)}\asymp\Vert f_a\Vert_{D^p_{p-1}}\asymp\Vert f_a\Vert_{H^p}=1.$ Furthermore, 
 it is clear that $f_a\to 0$, as $a\to1$ uniformly on compact subsets of $\D$, and $H_\omega(f_a)=\widetilde{H_\omega}(f_a)$. 
Since $T:X_p\to Y_p$ is compact, $\omega\in\DD$  by Theorem~\ref{bounded Xp}.
So,
Lemma~\ref{le:Htildeequiv} implies that 
$$\|T(f_a)\|_{Y_p}\asymp\|T(f_a)\|_{\hlp}, \quad a\in (0,1).$$

 Therefore, by using \eqref{CoeffCompact} we have
\begin{equation}\begin{split}\label{eq:faco}
\|\Ho(f_a)\|^p_{Y_p}&\gtrsim\|\Ho(f_a)\|^p_{\hlp}=\sum_{n=0}^\infty (n+1)^{p-2}\left(\frac{\sum\limits_{k=0}^\infty\widehat{f_a}(k)\omega_{k+n}}{2(n+1)\om_{2n+1}} \right)^p
\\ &\gtrsim (1-a^2)\sum_{n=0}^\infty \frac1{(n+1)^{2}}\left( \sum\limits_{k=0}^\infty (k+1)^{2/p-1}a^k\,\frac{\om_{k+n}}{\om_{2n}}\right)^p
\\ &\ge (1-a)\sum_{n=0}^\infty \frac1{(n+1)^{2}}\left( \sum\limits_{k=0}^n (k+1)^{2/p-1}a^k\,\frac{\om_{k+n}}{\om_{2n}}\right)^p
\\ &\ge (1-a)\sum_{n=0}^\infty \frac1{(n+1)^{2}}\left( \sum\limits_{k=0}^n (k+1)^{2/p-1}a^k\right)^p
\\ &\ge (1-a)\sum_{n=0}^\infty \frac{a^{pn}}{(n+1)^{2}}\left( \sum\limits_{k=0}^n (k+1)^{2/p-1}\right)^p
\\ &\gtrsim (1-a)\sum_{n=0}^\infty a^{pn}=\frac{1-a}{1-a^p}\gtrsim 1,
\end{split}\end{equation}
so  using  Theorem~\ref{compactXp} we deduce  that $T: X_p\to Y_p$ is not a compact operator.
\end{proof}

\section{Hilbert type operators acting on $X_1$-spaces}\label{sec:p=1}

The first result of this section gives the equivalence of conditions (iii)-(vi) of Theorem~\ref{bounded X_1}.

\begin{lemma}\label{le:m1}
\label{reformulacion M1}
Let  $\om \in \DD$. Then, the following conditions are equivalent:
\begin{enumerate}
\item[\rm(i)] $K_{1,c}(\om)=\sup\limits_{a \in [0,1)}
        \frac{1}{1-a}\int_a^1 \om(t)\left(1+\int_0^t \frac{ds}{\omg(s)}\right)\,dt<\infty;$
      \item[\rm(ii)] $K_{1,d}(\om)= \sup\limits_{a \in [0,1)} \frac{\omg(a)}{1-a} \left(1+\int_0^a \frac{ds}{\omg(s)}\right)<\infty;$
      \item[\rm(iii)] $ M_1(\omega)=\sup \limits_{N \in \N} (N+1)\om_{2N}\sum\limits_{k=0}^{N}\frac{1}{(k+1)^2 \om_{2k}}<\infty$.
\end{enumerate}
Moreover,
\begin{equation}\label{m1equiv}
K_{1,c}(\om)\asymp K_{1,d}(\om)\asymp  M_1(\omega),
\end{equation}
and $\omega\in\Dd$ when
  $\omega$ satisfies any of the three previous conditions.
 \end{lemma}
 Observe that for any radial weight, $K_{1,c}(\om)<\infty$ holds if and only if $M_{1,c}(\om)<\infty$, and analogously $K_{1,d}(\om)<\infty$ if and only if $M_{1,d}(\om)<\infty$. This fact will be used repeatedly throughout the paper.
\begin{proof}
On the one hand, 
\begin{align*}
\frac{\omg(a)}{1-a} \left(1+\int_0^a \frac{ds}{\omg(s)}\right)& \leq \frac{1}{1-a} \int_a^1 \om(t)\left(1+\int_0^t \frac{ds}{\omg(s)}\right)\,dt \\
&\leq K_{1,c}(\om),
  \quad a\in [0,1),
\end{align*}
so (i)$\Rightarrow$(ii) and $K_{1,d}(\om)\lesssim K_{1,c}(\om)$ .

On the other hand assume that (ii) holds.  Since $\om \in \DD$, \cite[Lemma~3(ii)]{PPR20} (for $\nu(t)=1$) yields
\begin{align*}
 \frac{1}{1-a}\int_a^1 \om(t)\left(1+\int_0^t \frac{ds}{\omg(s)}\right)\,dt &\leq 
 K_{1,d}(\om) \left( \frac{1}{1-a}\int_a^1\frac{ \om(t) (1-t)}{\omg(t)} dt\right)
 \\ &\lesssim  K_{1,d}(\om),\quad 0<a<1,
\end{align*}
that is (i) holds and  $K_{1,c}(\om)\lesssim K_{1,d}(\om)$ .  Finally, by mimicking the proof of Lemma \ref{reformulacionMp},
\begin{align*}
 K_{1,d}(\om)\asymp  M_1(\omega),
\end{align*}
so (ii)$\Leftrightarrow$(iii) and \eqref{m1equiv} holds.

Next, 
for any $K>1$ and $r\in (0,1)$ 
\begin{equation*}\begin{split}
M_{1,d}(\om) &\ge 
\frac{K\omg\left(1-\frac{1-r}{K}\right)}{1-r}
\int_r^{1-\frac{1-r}{K}} \frac{ds}{\omg(s)}
\\ &\ge 
(K-1)\frac{\omg\left(1-\frac{1-r}{K}\right)}{\omg\left( r\right)},
\end{split}\end{equation*}
that is 
$$\omg\left( r\right)\ge \frac{K-1}{M_{1,d}(\om)}\omg\left(1-\frac{1-r}{K}\right),\quad 0<r<1,$$
so taking $K>M_{1,d}(\om)+1$, we get $\omega\in\Dd$. This finishes the proof.

\end{proof}

The following result will be used to prove the  equivalence (ii)$\Leftrightarrow$(iii) of Theorem~\ref{bounded X_1}.

\begin{proposition}
\label{1-medidas de Carleson para HL(1)}
Let $\mu$ be a finite positive Borel measure  on $[0,1)$ and $X_1\in\{H(\infty,1), HL(1)\}$. Then $\mu$ is a $1$-Carleson measure for $X_1$ 
if and only if $\mu$ is a classical Carleson measure. Moreover,
$$\| I_d\|_{X_1\to L^1(\mu)}\asymp \sup_{a\in [0,1)}\frac{\mu([a,1))}{1-a}.$$

\end{proposition}
\begin{proof}
If $\mu$ is a $1$-Carleson measure for $X_1$, then  by \eqref{fejer} and \eqref{eq:HpDpp<2}, $\mu$ is a $1$-Carleson measure for $H^1$. So, by \cite[Theorem~9.3]{Duren} and its proof,
$\mu$ is a classical Carleson measure and $$\sup_{a\in [0,1)}\frac{\mu([a,1))}{1-a}\lesssim \| I_d\|_{X_1\to L^1(\mu)}.$$ Conversely, if $\mu$ is a classical Carleson measure, two integration by parts yield
$$ \int_{\D}  |f(z)|d \mu (z)=\int_0^1 |f(t)|d \mu (t)
\le \int_0^1 M_\infty(t,f)\,d \mu (t)
 \lesssim \| f\|_{H(\infty,1)}
\sup_{a\in [0,1)}\frac{\mu([a,1))}{1-a} 
.$$
This inequality, together with Lemma~\ref{HLP C LP}, finishes the proof.
\end{proof}

We introduce some more notation  to prove Theorem~\ref{bounded X_1}.
For any $C^\infty$-function $\Phi:\mathbb{R}\to\C$ with compact support,
define the polynomials
    \begin{equation*}
    W_n^\Phi(z)=\sum_{k\in\mathbb
    Z}\Phi\left(\frac{k}{n}\right)z^{k},\quad n\in\N.
    \end{equation*}

 A particular case of the previous construction is useful for our purposes.
 Some properties of these polynomials have been gathered in the next lemma, see \cite[Section~2]{JevPac98} or \cite[p. 143--144]{Pabook2} for a proof.

\begin{letterlemma}\label{pr:cesaro}
Let $\Psi:\mathbb{R}\to\mathbb{R}$ be a $C^\infty$-function such that $\Psi\equiv 1$ on $(-\infty,1]$, $\Psi\equiv 0$ on $[2,\infty)$ and $\Psi$ is decreasing and positive on $(1,2)$. Set $\psi(t)=\Psi\left(\frac{t}{2}\right)-\Psi(t)$ for all $t\in\mathbb{R}$. Let $V_{0}(z)=1+z$
and
    \begin{equation*}\label{vn}
    V_{n}(z)=W^\psi_{2^{n-1}}(z)=\sum_{j=0}^\infty
    \psi\left(\frac{j}{2^{n-1}}\right)z^j=\sum_{j=2^{n-1}}^{2^{n+1}-1}
    \psi\left(\frac{j}{2^{n-1}}\right)z^j,\quad n\in\N.
    \end{equation*}
 
Then,
    \begin{equation}
    \label{propervn1}
    f(z)=\sum_{n=0}^\infty (V_{n}\ast f)(z),\quad z\in\D,\quad f\in\H(\D),
    \end{equation}
    and for each $0<p<\infty$ there exists a constant $C=C(p,\Psi)>0$ such that
    \begin{equation}
    \label{propervn2}
   \|V_{n}\ast f\|_{H^p}\le C\|f\|_{H^p},\quad f\in H^p, \quad n \in \N.
    \end{equation}
    In addition 
    \begin{equation}
    \label{propervn3}
    \|V_n\|_{H^p}\asymp 2^{n(1-1/p)}, \quad 0< p<\infty.
    \end{equation}
\end{letterlemma}

Let us denote $f_r(z)=f(rz)$, $z\in\D$, $r\in (0,1)$.
Now we are ready to prove  the main theorem of this section.

\begin{Prf}\textit{Theorem \ref{bounded X_1}.}
First of all, recall  that $M_{1,c}(\om)<\infty$ if and only if $K_{1,c}(\om)<\infty$ and analogously $M_{1,d}(\om)<\infty$ if and only if $K_{1,d}(\om)< \infty$, so that the equivalences (iii)$\Leftrightarrow$(iv)$\Leftrightarrow$(v)$\Leftrightarrow$(vi) follow from Lemma~\ref{le:m1}.
The equivalence between (ii) and (iii) is a consequence of \cite[Theorem~9.3]{Duren} when $X_1=H^1$, \cite[Theorem~2.1]{Vinogradov} when $X_1=D^1_0$ and Proposition \ref{1-medidas de Carleson para HL(1)} when $X_1\in \{H(\infty,1),HL(1)\}$.

 (i)$\Rightarrow$(iii). In order to obtain both conditions, $\omega\in\DD$ and $M_{1,c}(\omega)<\infty $, we are going to deal with functions  $f \in \H(\D)$ such that $\fg(n) \geq 0$ for all $n \in \N \cup \{0\}$, so   it is enough to prove the result for $T= H_{\om }$.

 {\bf{First Step.}} Let us prove $\omega\in\DD$.
Bearing in mind
Lemma~\ref{HLP C LP} and \eqref{eq:HpDpp<2}
\begin{equation}\begin{split}\label{eq:necp1}
\sum\limits_{n=0}^{\infty}\frac{\om_{n+N}}{(n+1)^2 \om_{2n+1}} \left(\sum\limits_{k=0}^{N}\fg (k)\right) & 
\leq 
\|\Ho (f)\|_{H(\infty,1)} \lesssim  \|\Ho (f)\|_{Y_1}
\lesssim \| f\|_{X_1}\lesssim
 \| f \|_{D_0^1}, \quad N\in \N,
\end{split}
\end{equation}
for any $f\in \H(\D)$ such that $\fg(n)\ge 0$,  $n\in\N\cup\{0\}$. Next, 
for each $N \in \N$, consider the test functions $f_{\a, N}(z)=\sum\limits_{k=0}^N (k+1)^{\a}z^k,\, \a>0$. 
Set $M \in \N$ such that $2^M <N \leq 2^{M+1}$. Then, bearing in mind \eqref{propervn1}, $$(f_{\a, N}')_s(z)=\sum\limits_{n=0}^\infty (V_{n}\ast (f_{\a,N}')_s)(z)=\sum\limits_{n=0}^M (V_{n}\ast (f_{\a,N}')_s)(z),$$ which together with \cite[Lemma~3.1]{MatPav}, \cite[Lemma~5.4]{Pabook2} and \eqref{propervn3} gives
\begin{align*}
\| f_{\a, N}\|_{D_0^1}&\le\int_0^1 M_1(s, f_{\a,N}')ds \leq \sum\limits_{n=0}^M \int_0^1 \| V_{n}\ast (f_{\a,N}')_s \|_{H^1}ds \\
&\lesssim \sum\limits_{n=0}^M \int_0^1  s^{2^{n-1}} 2 ^{n(\a +1)} \| V_n\|_{H^1} ds \asymp  \sum\limits_{n=0}^M 2^{n\a}\asymp 2^{M\a} \asymp (N+1)^{\a}.
\end{align*}
Testing the functions $f_{\a, N}$ in 
 \eqref{eq:necp1},  
$\sup\limits_{N\in\N}(N+1)\sum\limits_{n=0}^{\infty}\frac{\om_{n+N}}{(n+1)^2 \om_{2n+1}} <\infty. $ Therefore,  there exists $C=C(\om)>0$
$$ \frac{\om_{8N}}{\om_{12N}} \lesssim \frac{\om_{8N}}{\om_{12N}} (N+1)\sum\limits_{n=6N}^{7N}\frac{1}{(n+1)^2} \leq (N+1) \sum\limits_{n=6N}^{7N}\frac{\om_{n+N}}{(n+1)^2 \om_{2n+1}} \leq C.$$
So, arguing as in the first step proof of Proposition \ref{cond nec Xp},  $\om \in \DD.$

 {\bf{Second Step.}} We will prove 
 $M_{1,c}(\omega)<\infty $. Let us consider the test functions  $f_a(z)=\frac{1-a^2}{(1-az)^2}$, $a\in (0,1)$.
A calculation shows that $ \|f_a\|_{D^1_0}\asymp 1$, $a\in (0,1)$.
 Then, by Lemma \ref{HLP C LP} and \eqref{eq:HpDpp<2},
$$\|H_\omega\|_{X_1\to Y_1}\gtrsim \sup_{a\in (0,1)}  \|\Ho(f_a)\|_{Y_1} \gtrsim 
\sup_{a\in (0,1)} \|\Ho(f_a)\|_{L^1_{[0,1)}}.$$
Consequently, using that $\omega\in\DD$
and  mimicking the proof $(4.2)$ of \cite[Theorem~2]{PelRosa21}, we get $M_{1,c}(\omega)<\infty.$

Now let us prove (iv) $\Rightarrow$ (i). 
Firstly, observe that the condition
$ M_{1,c}(\omega)<\infty$ implies $ K_{1,c}(\omega)<\infty$ so that $\widetilde{\omega}(t)=\frac{\omg(t)}{1-t}$ is bounded on $[0,1)$. So, using Lemma ~\ref{HLP C LP} and 
\eqref{eq:intparts}, 
$$\int_0^1 M_\infty(t,f)\omega(t)\,dt\lesssim \int_0^1 M_\infty(t,f)\widetilde{\omega}(t)\,dt\lesssim 
\| f\|_{H(\infty,1)}\lesssim \| f\|_{X_1},$$
that is $H_\omega(f)\in \H(\D)$ for any $f\in X_1$. Secondly, 
by \eqref{eq:HpDpp<2} and Lemma~\ref{HLP C LP},
 it is enough  to prove the inequality
 $$ \| \Ho (f) \|_{D^1_0}\lesssim \|f\|_{H(\infty,1)},\quad f\in H(\infty,1),$$
 to end the proof.
Indeed,
\begin{align*}
\| \Ho (f) \|_{D^1_0}&\le\int_0^1 M_1(s, \Ho (f)') ds \leq \int_0^1 \left( \int_0^1  |f(t)| \, \om (t) M_1(s, G^\omega_t ) dt \right) ds
\\&= \int_0^1  |f(t)| \om (t) \left(\int_0^1  M_1(s,  G^\omega_t ) ds \right) dt .
\end{align*}
Then by \eqref{eq:Gwt} and \cite[Lemma~B]{PelRosa21}
 $$M_1(s,G^\omega_t) 
  \asymp 1+\int_0^{st}\frac{dx}{\omg (x)(1-x)}, \quad 0\leq s,t<1.$$
Bearing in mind that $ M_{1,c}(\omega)<\infty$ implies $ K_{1,c}(\omega)<\infty$ and applying Proposition~\ref{1-medidas de Carleson para HL(1)}, the measure $\mu_\om$ defined as $d\mu_{\om}(z)= \om(z)\left(1+\int_0^{|z|} \frac{ds}{\omg(s)}\right)\,\chi_{[0,1)}(z)\,  dA(z)$ is a $1$-Carleson measure for $H(\infty , 1)$, so by Tonelli's theorem,
\begin{align}\label{eq:D10HL1}
 \| \Ho (f) \|_{D^1_0} &\lesssim  \int_0^1  |f(t)| \, \om (t) \left(  1+  \int_0^1 \left(  \int_0^{st} \frac{dx}{\omg (x)(1-x)} \right) ds \right) dt 
 \nonumber
 \\&= \int_0^1  |f(t)| \, \om (t) \left(  1+  \int_0^t \frac{(1-\frac{x}{t})}{\omg (x)(1-x)} dx  \right) dt \nonumber
 \\ & \leq \int_0^1  |f(t)| \, \om (t) \left(  1+  \int_0^t \frac{dx}{\omg (x)} \right) dt \lesssim
\|f\|_{H(\infty,1)}.
\end{align}

This finishes the proof.
\end{Prf}

\subsection{$\mathbf{H_\omega: X_1\to Y_1}$ versus $\mathbf{H_\omega: L^1_{[0,1)}\to Y_1}$. }

Firstly, we will study the boundedness of $T: L^1_{[0,1)}\to Y_1$, $T\in \{H_\omega,\widetilde{\Ho}\}$,  $Y_1\in \{H(\infty,1),H^1, D^1_{0}, HL(1)\}$.

\begin{theorem}
\label{bounded L1}
Let  $\om$ be a radial weight, let $Y_1\in \{H(\infty,1),H^1, D^1_{0}, HL(1)\}$ and let $T \in \{ \Ho, \widetilde{\Ho}\}$.  Then the following statements are equivalent:
\begin{itemize}
    \item[(i)] $T: L^1_{[0,1)}\to Y_1$ is bounded; 
    \item[(ii)] $\om \in \DDD$ and  
    $m_1(\om)= \esssup\limits_{t \in [0,1)}
         \om(t)\left(1+\int_0^t \frac{ds}{\omg(s)}\right)<\infty.$
    \item[(iii)] $\om \in \DD$ and  
    $m_1(\om)= \esssup\limits_{t \in [0,1)}
         \om(t)\left(1+\int_0^t \frac{ds}{\omg(s)}\right)<\infty.$    
\end{itemize}

\end{theorem}
\begin{proof}
(i)$\Rightarrow$(ii). If (i) holds, then  $\om \in \DDD$ by Lemma~\ref{HLP C LP} and Theorem~\ref{bounded X_1}. Next,
using Lemma~\ref{HLP C LP} again and making minor modifications in the proof of
\cite[Theorem~2]{PelRosa21} we get
$$\|T (f)\|_{Y_1}\gtrsim \|H_\om(f)\|_{L^1_{[0,1)}}\gtrsim \int_{0}^1 f(t)\om(t)\left(1+\int_0^t \frac{ds}{\omg(s)}\right)\,dt,$$
for any $f\ge 0$, $f\in  L^1_{[0,1)}$. So,
$$\int_{0}^1 f(t)\om(t)\left(1+\int_0^t \frac{ds}{\omg(s)}\right)\,dt\lesssim \|f\|_{L^1_{[0,1)}},\quad f\ge 0$$
which implies that $m_1(\om)<\infty$.

(ii)$\Rightarrow$(iii) is clear.

(iii)$\Rightarrow$(i). If (iii) holds, then $H_\omega(f)\in \H(\D)$ for any $f\in L^1_{[0,1)}$ and arguing as in \eqref{eq:D10HL1}
\begin{align*}
 \| \Ho (f) \|_{D^1_0} &\lesssim  \int_0^1  |f(t)| \, \om (t) \left(  1+  \int_0^t \frac{dx}{\omg (x)} \right) dt \lesssim
 \| f\|_{L^1_{[0,1)}}.
\end{align*}
This together with \eqref{eq:HpDpp<2} and Lemma~\ref{HLP C LP} gives that $H_{\om }: L^1_{[0,1)}\to Y_1$ is bounded. This finishes the proof.
\end{proof}

Joining Theorem~\ref{bounded X_1}, Theorem~\ref{bounded L1} and Lemma~\ref{HLP C LP} we deduce the following.

\begin{corollary}\label{co:goodweightsp1}
Let  $\om$ be a radial weight, $X_1,Y_1\in \{H(\infty,1),H^1, D^1_{0}, HL(1)\}$ and let $T\in \{\Ho, \widetilde{\Ho}\}$. If
$\omega$ satisfies the condition \eqref{eq:regularity}, then the following statements are equivalent:
\begin{itemize}
\item[(i)] $T:L^1_{[0,1)}\to Y_1$ is bounded;
\item[(ii)] $T: X_1 \to Y_1$ is bounded;
\item[(iii)] $\om \in \DD$ and $M_{1,c}(\om)<\infty$;
\item[(iv)] $\om \in \DDD$ and $M_{1,c}(\om)<\infty$;
\item[(v)] $\om \in \DD$ and $m_1(\om)<\infty$;
\item[(vi)] $\om \in \DDD$ and $m_1(\om)<\infty$.
\end{itemize}
\end{corollary}
\begin{proof}
(i)$\Rightarrow$(ii) follows from Lemma~\ref{HLP C LP}, and (ii)$\Leftrightarrow$(iii)$\Leftrightarrow$(iv) were proved in Theorem \ref{bounded X_1}.
Next, since $M_{1,c}(\om)<\infty$ implies $K_{1,c}(\om)<\infty$, (iii)$\Rightarrow$(v) is a byproduct of
the hypothesis \eqref{eq:regularity}. Finally, the equivalences (v)$\Leftrightarrow$(vi)$\Leftrightarrow$(i) follow from  Theorem~\ref{bounded L1}.
This finishes the proof.
\end{proof}

A similar comparison between the conditions $M_{1,c}(\om)<\infty$ and $m_1(\om)<\infty$, to that  made for the conditions
$M_{p,c}(\om)<\infty$ and $m_p(\om)<\infty$, $1<p<\infty$, can also be considered. The following result shows that they are not equivalent.

\begin{corollary}\label{co:comparisonL1-X1}
Let   $X_1,Y_1\in \{H(\infty,1),H^1, D^1_{0}, HL(1)\}$.
For each radial weigth $\nu$ such that $ Q: L^1_{[0,1)}\to Y_1$ is bounded, where $Q\in\{ H_\nu,\widetilde{H_\nu}\}$, there is a radial weight 
$\omega$ such that $$\widehat{\omega}(t)\asymp \widehat{\nu}(t),\quad t\in [0,1),$$ 
$\omega\in L^\infty_{[0,r_0]}$ for any $r_0\in(0,1)$,  $T: X_1\to Y_1$ is bounded and 
 $T: L^1_{[0,1)}\to Y_1$ is not bounded. Here $T\in\{ H_\omega,\widetilde{H_\omega}\}$.
\end{corollary} 
\begin{proof}
Since $ Q: L^1_{[0,1)}\to Y_1$ is bounded, by Theorem~\ref{bounded X_1}, $\nu\in\DDD$ and $M_{1,c}(\nu)<\infty$. Now, by Lemma~\ref{le:cont}
and its proof,  there is  
a radial weight $\omega$ such that 
$\widehat{\omega}(t)\asymp \widehat{\nu}(t)$, $\omega\in L^\infty_{[0,r_0]}$ for any $r_0\in(0,1)$ and $\omega\not\in L^{\infty}_{[0,1)}$. So, $m_1(\omega)=\infty$ and by Theorem~\ref{bounded L1},
$T: L^1_{[0,1)}\to Y_1$ is not bounded. Moreover, $\omega\in\DDD$ and $M_{1,c}(\om)<\infty$ because $\nu$ satisfies both properties, consequently
$T: X_1\to Y_1$ is bounded.
\end{proof}

\subsection{Compactness of Hilbert-type operators on  $X_1$-spaces. }

\begin{lemma}\label{Conv compactX1}
Let   $\om \in \DDD$ such that $M_{1,d}(\omega)<\infty$. Let
 $\{f_k\}_{k=0}^\infty\subset X_1  \in \{ H(\infty,1),D^1_{0}, H^1, HL(1)\}$ such that $\sup\limits_{k\in\N} \|f_k\|_{X_1}<\infty$ and $f_k\to 0$ uniformly on compact subsets of $\D$.
 Then the following statements hold:
\begin{enumerate}
\item[\rm(i)] $\int_0^1 |f_k(t)|\om(t)dt\to 0$ when $k\to\infty$.
\item[\rm(ii)] If $T\in\{H_\omega, \widetilde{H_\omega}\}$, then $T(f_k)\to 0$ uniformly on compact subsets of $\D$.
\end{enumerate}
\end{lemma}
\begin{proof}
Firstly, let us prove that 
\begin{equation}\label{eq:compact1}
\lim_{a\to 1^-} \frac{\widehat{\omega}(a)}{1-a}=0.
\end{equation}
Since $M_{1,d}(\omega)<\infty$, then 
\begin{equation*}\begin{split}
\int_0^{\frac{1+a}{2}} \frac{ds}{\omg(s)} &\ge \left(M_{1,d}(\omega)\right)^{-1}\int_0^{\frac{1+a}{2}} \left(\int_{0}^s\frac{dt}{\omg(t)}\right) \frac{ds}{1-s}
\\ &\ge \left(M_{1,d}(\omega)\right)^{-1}\int_{\frac{1}{2}}^{\frac{1+a}{2}} \left(\int_{0}^s\frac{dt}{\omg(t)}\right) \frac{ds}{1-s}
\\ &\ge \left(M_{1,d}(\omega)\right)^{-1} \left(\int_{0}^{\frac{1}{2}}\frac{dt}{\omg(t)}\right) \int_{\frac{1}{2}}^{\frac{1+a}{2}}  \frac{ds}{1-s}
\\ & = \left(M_{1,d}(\omega)\right)^{-1}\left(\int_{0}^{\frac{1}{2}}\frac{dt}{\omg(t)}\right) \log\frac{1}{1-a}, \quad 0<a<1.
\end{split}\end{equation*}
So $\lim_{a\to 1^-} \int_0^{a} \frac{ds}{\omg(s)}=\infty$, and then using again the condition $M_{1,d}(\omega)<\infty$, \eqref{eq:compact1} holds.

From now on, the proof follow the lines of Lemma~\ref{Conv compact}.
Let $\varepsilon>0$.  
By \eqref{eq:compact1}  there exists $0<\rho_0<1$ such that $ \widetilde{\om}(t) < \varepsilon $ for any $t\in [\rho_0,1)$. Moreover,  there exists $k_0$ such that for every $k\ge k_0$ and $z\in M=\overline{D(0,\rho_0)}$, $|f_k(z)|<\varepsilon$. Then, by Lemma \ref{le:descriptionDDD}  and  \eqref{eq:intparts} 
\begin{align*}
\int_0^1|f_k(t)|\om(t)dt& \leq |f_k(0)|\omg (0) + \int_0^1 M_{\infty}(t,f_k)\widetilde{\om}(t)dt
\\&\lesssim
 \int_0^{\rho_0} M_{\infty}(t,f_k)\widetilde{\om}(t)dt + \int_{\rho_0}^1 M_{\infty}(t,f_k)\widetilde{\om}(t)dt
\\  &\le \varepsilon \left(\int_0^{1}\widetilde{\om}(t)dt + \sup_{k\in\N} \|f_k\|_{X_1}\right)
=C\varepsilon,
\end{align*}
where in the second to last  step we have used Lemma \ref{HLP C LP}.

The proof of (ii) is analogous to that of Lemma~\ref{Conv compact} so we omit its proof.

\end{proof}

Using the previous lemma and Theorem~\ref{bounded X_1} we obtain the following by mimicking the proof of Theorem~\ref{compactXp}.

\begin{theorem}\label{compactX1}
Let $\omega$ be a radial weight and $X_1, Y_1 \in \{ H(\infty,1),D^1_{0}, H^1, HL(1)\}$ and let $T\in\{H_\omega, \widetilde{H_\omega}\}$. Then, the following assertions are equivalent:
\begin{enumerate}
\item[\rm(i)]
  $T: X_1 \to Y_1$ is compact;
  \item[\rm(ii)] For every sequence $\{f_k\}_{k=0}^\infty\subset X_1 $ such that $\sup\limits_{k\in\N} \|f_k\|_{X_1}<\infty$ and $f_k\to 0$ uniformly on compact subsets of $\D$,   $\lim\limits_{k\to\infty} \|T(f_k)\|_{Y_1} =0$.
  \end{enumerate}
\end{theorem}

\begin{theorem}\label{pr:nocompactx1}
Let $\omega$ be a radial weight and  $X_1,Y_1\in \{ H(\infty,1),D^1_{0}, H^1, HL(1)\}$, and let $T\in \{\Ho, \widetilde{\Ho}\}$. Then, 
$T:X_1\to Y_1$ is not compact.
\end{theorem}

\begin{proof}
The proof is analogous to that of Theorem~\ref{pr:nocompactxp}, so we provide a sketch of the proof.
Assume that $T:X_1\to Y_1$ is compact. For each $0<a<1$, set
$
f_a(z)=\frac{1-a^2}{(1-az)^2},\, z\in\D. $
A calculation shows that $\sup_{a\in (0,1)}\Vert f_a\Vert_{X_1}\asymp 1$ and
 $f_a\to 0$, as $a\to1$ uniformly on compact subsets of $\D$.
 Moreover, since $T(f_a)$ has non-negative Taylor coefficients, 
    $$\|T(f_a)\|_{Y_1}\gtrsim \|\Ho(f_a)\|_{H(\infty,1)}\asymp \|\Ho(f_a)\|_{HL(1)}\gtrsim 1,\quad a\in (0,1),$$
    where  the last inequality follows taking $p=1$ in \eqref{eq:faco}.
So, using  Theorem~\ref{compactX1} we deduce  that $T: X_1\to Y_1$ is not a compact operator.
\end{proof}

\section{Hilbert-type operators acting on $H^{\infty}$}\label{sec:infty}

We will prove a  result which includes Theorem~\ref{th:H infty Bloch intro}. With this aim 
we need some more notation. 
 The space $Q\sb
p$,\, $0\le p<\infty$, consists of those $f\in H(\D)$ such that 
\[
\| f\|^2_{Q_p}=|f(0)|^2+\sup\sb{a\in\D} \int\sb\D |f'(z)|\sp 2 (1-|\varphi_a(z)|^2)^p\,dA(z)<\infty \,,
\]
where $\varphi_a(z)=
 \frac{a-z}{1-\overline{a} z}, \; z, a
\in\D.$
If $p>1$, $Q_p$ 
coincides with the Bloch space $\B$.
 The space $Q\sb 1$
coincides with $BMOA$ (see, e.~g., \cite[Theorem~5.2]{GiBMO}). However, if
$0<p<1 $, $Q\sb p$ is a proper subspace of $BMOA $ \cite{X2}. The space $Q_0$ reduces to the classical Dirichlet
space $\mathcal{D}$.

We recall that
\begin{equation}\begin{split}\label{Qp:embeddings}
& Q_p\subsetneq \BMOA\subsetneq \B,\quad 0<p<1.
\\ & H^\infty \subsetneq \BMOA\subsetneq \B,
\end{split}\end{equation}
however if $0<p<1$, $H^\infty\not\subset Q_p$, and $Q_p \not\subset H^\infty$, see \cite{X2}.

Moreover, $HL(\infty)\subsetneq  Q_p$.  
 This embedding might have been proved in some previous paper, however we include a short direct proof for the sake of completeness.
\begin{lemma}\label{le:HlinftyQpembedding}
Let $0<p\le \infty$, then $HL(\infty)\subsetneq Q_p$ and 
$$\|f\|_{Q_p}\lesssim \|f\|_{HL(\infty)},\quad f\in\H(\D).$$
\end{lemma}
\begin{proof}
Let $f\in HL(\infty)$, then
\begin{align*}
M_2^2(\rho,f^\prime)= \sum_{n=1}^\infty n^2|\fg (n)|^2 \rho^{2n-2}\le \|f\|_{HL(\infty)}^2\sum_{n=1}^\infty\rho^{2n-2}=\frac{\|f\|_{HL(\infty)}^2}{1-\rho^2},
\end{align*}
So for any $0<p<\infty$,
\begin{align*}
\sup_{a\in\D}\int_\D (1-|\varphi_a(z)|^2)^p|f^\prime(z)|^2\,dA(z) &\lesssim
\sup_{a\in\D} \int_0^1 \left(\frac{(1-|a|)(1-s^2)}{(1-|a|s)^2}\right)^p M_2^2(s,f^\prime)\,ds
\\ &\le \|f\|_{HL(\infty)}^2 \sup_{a\in\D}\int_0^1 \frac{(1-|a|)^p(1-s^2)^{p-1}}{(1-|a|s)^{2p}}\, ds
\\ &\lesssim \|f\|_{HL(\infty)}^2,
\end{align*}
so $\|f\|_{Q_p}\lesssim\|f\|_{HL(\infty)}$. 
The lacunary series $f(z)=\sum_{k=0}^\infty  2^{-k}\log{(k+2)}z^{2^k}\in \bigcap\limits_{0<p} Q_p\setminus HL(\infty)$. This finishes the proof.
\end{proof}

Now we will prove the main result of this section, which is an extension of Theorem~\ref{th:H infty Bloch intro}.
\begin{theorem}
\label{th:H infty Bloch}
Let $\om $ be a radial weight and let $T \in \{H_\omega, \widetilde{H_\omega}\}$. Then, the following statements are equivalent:
\begin{itemize}
\item[(i)] $T: H^{\infty}\to HL(\infty)$ is bounded;
\item[(ii)]  $T: H^{\infty}\to Q_p$ is bounded for $0<p<1$;
\item[(iii)] $T: H^{\infty}\to \BMOA$ is bounded;
\item[(iv)]  $T: H^{\infty}\to \B$ is bounded;
\item [(v)] $\om \in \DD$.
\end{itemize}
\end{theorem}
\begin{Prf} {\em{Theorem~\ref{th:H infty Bloch}.}}

By Lemma~\ref{caract. pesos doblantes},
\begin{equation*}\begin{split}
\|\Ho(f)\|_{HL(\infty)} &= \sup\limits_{k\in \N\cup\{0\}} (k+1)\left| \frac{\int_0^1 f(t)t^k \om(t)dt}{2(k+1)\om_{2k+1}}\right|
\\ &\le \sup\limits_{k\in \N\cup\{0\}} (k+1) \frac{\int_0^1 |f(t)|t^k \om(t)dt}{2(k+1)\om_{2k+1}}
\\ &=\|\widetilde{\Ho}(f)\|_{HL(\infty)}
\lesssim \|f\|_{H^{\infty}},
\end{split}\end{equation*}
so (v)$\Rightarrow $ (i).
The implications (i)$\Rightarrow $(ii)$\Rightarrow $(iii)$\Rightarrow $(iv)
 follow from
\eqref{Qp:embeddings} and Lemma~\ref{le:HlinftyQpembedding}.

\par The implication (iv)$\Rightarrow $(v)  was proved in \cite[Theorem~1]{PelRosa21}, and this finishes the proof.

\end{Prf}

\par It is worth mentioning that for
$f(z)=\log\frac{1}{1-z}\in HL(\infty)$ and $\omega$  a radial weight,
\begin{equation*}\begin{split}
H_\omega(f)'(x)&=\sum_{n=1}^\infty\frac{n}{2(n+1)\omega_{2n+1}}\left(\sum_{k=1}^\infty \frac{\om_{n+k}}{k} \right)x^{n-1}
\\ & \ge \sum_{n=1}^\infty\frac{n}{2(n+1)}\left(\sum_{k=1}^n\frac{1}{k} \right)x^{n-1}
\\ & \gtrsim \sum_{n=1}^\infty \log(n+1) x^{n-1},
\end{split}\end{equation*}
so $H_\omega(f)\notin \B$.
So, the space $H^\infty$ cannot be replaced by $HL(\infty)$ and by any $Q_p$ space, $0<p<\infty$, in the statement of
Theorem~\ref{th:H infty Bloch}. That is,
the remaining cases for $p=\infty$, analogous to those of 
 Theorems~\ref{bounded Xp} and \ref{bounded X_1}, which do not appear in Theorem~\ref{th:H infty Bloch}, simply do not hold for any radial weight.

\vspace{1em}
Finally,  we will prove the analogous result to Theorem~\ref{pr:nocompactxp} for $p=\infty$.
\begin{theorem}\label{th:nocompactoHinfty}
Let $\omega$ be a radial weight and let $T \in \{H_\om , \widetilde{H_\omega}\}$. Then $T:H^{\infty}\to Y_\infty$ is not a compact operator, where $Y_\infty\in\{Q_p,\B,\BMOA, HL(\infty )\}$, $0<p<1$.
\end{theorem}

We need the following result, whose proof can be obtained by mimicking the proof  Theorem~\ref{compactXp}.

\begin{theorem}\label{compactcharBii}
Let $\omega$ be a radial weight and $T\in\{H_\omega, \widetilde{H_\omega}\}$. Then, the following assertions are equivalent:
\begin{enumerate}
\item[\rm(i)]
  $T: H^{\infty}\to \B$ is compact;
  \item[\rm(ii)] For every sequence $\{f_k\}_{k=0}^\infty\subset H^\infty$ such that $\sup\limits_{k\in\N} \|f_k\|_{H^\infty}<\infty$ and $f_k\to 0$ uniformly on compact subsets of $\D$,   $\lim\limits_{k\to\infty} \|T(f_k)\|_\B =0$.
  \end{enumerate}
\end{theorem}

\begin{Prf}{\em{Theorem~\ref{th:nocompactoHinfty}}.}
By \eqref{Qp:embeddings} and Lemma~\ref{le:HlinftyQpembedding}, it is enough to prove that $T: H^\infty\to\B$ is not compact.
Let consider for every $k\in\N$ the function $f_k(z)=z^k$, $z\in\D$. It is clear that $\|f_k\|_{H^\infty}=1$ for every $k\in\N$ and $f_k\to 0$ uniformly on compact subsets of $\D$. Since 
$$T(f_k)(z)=\sum_{n=0}^\infty \frac{\int_0^1 f_k(t)t^n\om(t)\,dt}{2(n+1)\om_{2n+1}}z^n=\sum_{n=0}^\infty \frac{\om_{n+k}}{2(n+1)\om_{2n+1}}z^n,$$  
for any $k\ge2$
\begin{align*}
\|T(f_k)\|_{\B}&\ge\sup_{x\in(0,1)}(1-x)\sum_{n=1}^\infty \frac{n\om_{n+k}}{2(n+1)\om_{2n+1}}x^{n-1} 
\ge \frac14\sup_{x\in(0,1)}(1-x)\sum_{n=1}^\infty \frac{\om_{n+k}}{\om_{2n}}x^{n-1}
\\ & \ge \frac14\sup_{x\in(0,1)}(1-x)\sum_{n=k}^{2k} \frac{\om_{n+k}}{\om_{2n}}x^{n-1}
\ge \frac14\sup_{x\in(0,1)}(1-x)\sum_{n=k}^{2k}x^{n-1}
\\ & \ge  \frac14\sup_{x\in(0,1)}(1-x)kx^{2k-1}
\ge \frac18\left(1-\frac{1}{2k}\right)^{2k-1}\ge \frac18\inf_{m\ge2}\left(1-\frac{1}{m}\right)^{m}\ge C>0
\end{align*}
so $\lim\limits_{k\to\infty}\|T(f_k)\|_\B\neq0$ and hence, by Theorem \ref{compactcharBii}, $T: H^\infty\to\B$ is not compact.
\end{Prf}

\par Before ending this section, we  briefly compare the action of the Hilbert-type operator $H_\omega$ and the   Bergman projection  \begin{equation*}
    P_\om(f)(z)=\int_{\D}f(\z) \overline{B^\om_{z}(\z)}\,\om(\z)dA(\z),
    \end{equation*}
    induced by a radial weight $\omega$.
As a consequence of  Theorem~\ref{th:H infty Bloch} and \cite[Theorem~1]{PR19}, the condition $\om \in \DD$ characterizes the boundedness of the operators $H_{\om}:H^{\infty}\to \B$ and $P_{\om}:L^{\infty}\to \B$.
Moreover,  $P_{\om}:L^{\infty}\to \B$ is bounded and onto if and only if $\om \in \DDD$ \cite[Theorem~3]{PR19}.
So, it is natural to think about the radial weights such that  the operator $H_{\om}:H^{\infty}\to \B$ is bounded and onto.
A straightforward argument proves
 there is no radial weights such that $H_{\om}:H^{\infty}\to \B$ satisfies both properties: 
If $H_{\om}:H^{\infty}\to \B$ is bounded, Theorem \ref{th:H infty Bloch} yields that $H_{\om}:H^{\infty}\to \BMOA$ is also bounded, so if $g\in \B\setminus\BMOA$, \textit{e.g.} $g(z)=\sum\limits_{k=0}^\infty z^{2^k}$, there does not exist $f\in H^\infty$ such that $\Ho(f)=g$. Consequently, 
$H_{\om}:H^{\infty}\to \B$ is not surjective.

\section{Comparisons and reformulations of the $M_{p,c}$-conditions. }\label{sec:q<p}

In order to prove Theorem~\ref{th:q<p} we will study the relationship between  some of the conditions 
which describe the boundedness of the Hilbert-type operators $H_\omega$ and $\widetilde{H_\omega}$ from $X_p$ to $Y_p$,
and $X_q$ to $Y_q$.

\begin{Prf}{\em{Theorem~\ref{th:q<p}.}}

Firstly, assume $1<q<p<\infty$. Since $T:X_q\to Y_q$ is bounded, Theorem~\ref{bounded Xp} yields $\omega\in\DD$ and $M_{q,c}(\om)<\infty$, and as a consequence, $K_{q,c}(\om)<\infty$.  By using Lemma \ref{caract. pesos doblantes}, 
\begin{equation}\label{eq:r1}
1+\int_0^r \frac{ds}{\omg(s)^q}\gtrsim \frac{1-r}{\omg(r)^q},\quad 0\le r<1.
\end{equation}
On the other hand, H\"older's inequality with exponents $x=\frac{q'}{p'}>1$ and $x'=\frac{x}{x-1}$, implies
\begin{equation}\label{eq:q<p1}
 \left(\int_r^1 \left(\frac{\omg(s)}{1-s}\right)^{p'}\,ds \right)^{\frac{1}{p'}}
\le \left(\int_r^1 \left(\frac{\omg(s)}{1-s}\right)^{q'}\,ds \right)^{\frac{1}{q'}}(1-r)^{\frac{1}{p'}-\frac{1}{q'}},
\quad 0\le r<1.
\end{equation}
Moreover, 
\begin{equation}\label{eq:q<p2} \left( 1+\int_0^r \frac{1}{\omg(s)^p}\,ds \right)^{\frac{1}{p}}\lesssim 
\omg(r)^{\frac{q}{p}-1}
\left(1+\int_0^r \frac{1}{\omg(s)^q}\,ds \right)^{\frac{1}{p}},
 \quad 0\le r<1. \end{equation}
So, the identity $\frac{1}{p'}-\frac{1}{q'}=\frac{1}{q}-\frac{1}{p}$,
 together with  \eqref{eq:r1}, \eqref{eq:q<p1} and \eqref{eq:q<p2} yield
\begin{equation*}\begin{split}
&\left(1+\int_0^r \frac{1}{\omg(t)^p} dt\right)^{\frac{1}{p}}
\left(\int_r^1 \left(\frac{\omg(t)}{1-t}\right)^{p'}\,dt\right)^{\frac{1}{p'}}
\\ & \lesssim \left( \frac{1-r}{\omg(r)^q}\right)^{\frac{1}{q}-\frac{1}{p}} \left(1+\int_0^r \frac{1}{\omg(s)^q}\,ds \right)^{\frac{1}{p}}\left(\int_r^1 \left(\frac{\omg(s)}{1-s}\right)^{q'}\,ds \right)^{\frac{1}{q'}}
\\ & \lesssim \left(1+\int_0^r \frac{1}{\omg(s)^q}\,ds \right)^{\frac{1}{q}}\left(\int_r^1 \left(\frac{\omg(s)}{1-s}\right)^{q'}\,ds \right)^{\frac{1}{q'}}.
\end{split}\end{equation*}
Consequently $K_{p,c}(\om)<\infty$, and by Theorem~\ref{bounded Xp}, $T:X_p\to Y_p$ is bounded.

Assume that $q=1$, that is, $T:X_1\to Y_1$ is bounded. By Theorem~\ref{bounded X_1}, $\omega\in\DD$ and $M_{1,d}(\omega)<\infty$, so $K_{1,d}(\omega)<\infty$.
Then,
\begin{equation*}\begin{split}
\int_r^1 \left(\frac{\omg(s)}{1-s}\right)^{p'}\,ds
& \le K^{p'}_{1,d}(\omega) \int_r^1 
\left(1+\int_0^s \frac{1}{\omg(t)}\,dt \right)^{-p'} \,ds
\\ & \le 
 K^{p'}_{1,d}(\omega)(1-r)\left(1+\int_0^r \frac{1}{\omg(t)}\,dt \right)^{-p'}
\\ & \lesssim  K^{p'}_{1,d}(\omega) \frac{\omg(r)^{p'}}{(1-r)^{p'-1}}, \quad 0\le r<1,
\end{split}\end{equation*}
where in the last inequality we have used \eqref{eq:r1} with $q=1$. 
 Moreover,
$$ \left( 1+\int_0^r \frac{1}{\omg(s)^p}\,ds \right)^{\frac{1}{p}}\lesssim 
\frac{1}{\omg(r)^{1-\frac{1}{p}}}\left( 1+\int_0^r \frac{1}{\omg(s)}\,ds \right)^{\frac{1}{p}}\le
\frac{(1-r)^{\frac{1}{p}}}{\omg(r)} K_{1,d}(\omega)^{1/p}, \quad 0\le r<1.$$
So, $M_{p,c}(\om)<\infty$, and by Theorem~\ref{bounded Xp}, 
$T:X_p\to Y_p$ is bounded.
This finishes the proof.
\end{Prf}

Finally, we present two more conditions which characterize the radial weights $\omega$ such that 
$T: X_p\to Y_p$, $1<p<\infty$, is bounded, where $X_p,Y_p\in \{H(\infty,p),H^p, D^p_{p-1}, HL(p)\}$ and 
$T\in\{ \Ho, \widetilde{\Ho}\}$.

\begin{proposition}\label{pr:reforMp}
Let $\omega$ be a radial weight and $1<p<\infty$. Then, the following conditions are equivalent:
\begin{enumerate}
\item[\rm(i)] $\om \in \DD$ and 
$K_{p,d}(\omega)=\sup\limits_{0< r<1}\frac{\omg(r)}{(1-r)^{\frac{1}{p}}}\left(1+\int_0^r \frac{1}{\omg(s)^p}\,ds \right)^{\frac{1}{p}}<\infty$;
\item[\rm(ii)] $\om \in \DD$ and   $K_{p,e}(\omega)=\sup\limits_{0< r<1}\frac{(1-r)^{\frac{1}{p}}}{\omg(r)}
\left(\int_r^1 \left(\frac{\omg(s)}{1-s}\right)^{p'}\,ds \right)^{\frac{1}{p'}}<\infty$;
\item[\rm(iii)] $\om \in \DD$ and 
$
K_{p,c}(\om)=
\sup\limits_{0<r<1} \left(1+\int_0^r \frac{1}{\omg(t)^p} dt\right)^{\frac{1}{p}}
\left(\int_r^1 \left(\frac{\omg(t)}{1-t}\right)^{p'}\,dt\right)^{\frac{1}{p'}}<\infty.$
\end{enumerate}
\end{proposition}
\begin{proof}
Assume that (i) holds. A calculation shows that
$F(r)=(1-r)^{\kappa}\left( 1+\int_0^r \frac{ds}{\omg(s)^p}\, \right)$, with $\kappa=\frac{1}{K^p_{p,d}(\omega)}$, is non-decreasing in $[0,1)$. So,
\begin{equation*}\begin{split}
\int_r^1 \left(\frac{\omg(s)}{1-s}\right)^{p'}\,ds
& \le K^{p'}_{p,d}(\omega) \int_r^1 \frac{1}{1-s}\left(1+\int_0^s \frac{1}{\omg(t)^p}\,dt \right)^{-\frac{p'}{p}} \,ds
\\ & \le K^{p'}_{p,d}(\omega)F(r)^{-\frac{p'}{p}} \int_r^1 (1-s)^{\frac{\kappa p'}{p}-1} ds
\\ & \lesssim  K^{p'+p}_{p,d}(\omega)\left(1+\int_0^r \frac{1}{\omg(t)^p}\,dt \right)^{-\frac{p'}{p}}
\\ & \lesssim  K^{p'+p}_{p,d}(\omega) \frac{\omg(r)^{p'}}{(1-r)^{p'-1}}, \quad 0\le r<1,
\end{split}\end{equation*}
where in the last inequality we have used \eqref{eq:r1}. 
That is (ii) holds.

Now, assume that (ii) holds. Since $\om\in\DD$
\begin{equation*}\label{eq:r2}
\int_r^1 \left(\frac{\omg(s)}{1-s}\right)^{p'}\,ds \gtrsim \frac{\omg(r)^{p'}}{(1-r)^{p'-1}},\quad 0\le r<1.
\end{equation*}
Moreover,
$H(r)=(1-r)^{-\eta}\int_r^1 \left(\frac{\omg(s)}{1-s}\right)^{p'}\,ds $, with $\eta=\frac{1}{K^{p'}_{p,e}(\omega)}$, is non-increasing in $[0,1)$. 
So,
\begin{equation*}\begin{split}
 1+\int_0^r \frac{ds}{\omg(s)^p}\, &
 \le 1+ K^{p}_{p,e}(\omega) \int_0^r \frac{1}{1-s}\left(\int_s^1 \left(\frac{\omg(t)}{1-t}\right)^{p'}\,dt \right)^{-\frac{p}{p'}}\,ds
 \\ &  \le 1+ K^{p}_{p,e}(\omega) H(r)^{-\frac{p}{p'}} \int_0^r \frac{1}{(1-s)^{1+\eta\frac{p}{p'}}} ds
 \\ & \lesssim  K^{p+p'}_{p,e}(\omega)\left(\int_r^1 \left(\frac{\omg(t)}{1-t}\right)^{p'}\,dt \right)^{-\frac{p}{p'}}
 \\ & \lesssim  K^{p+p'}_{p,e}(\omega) \frac{1-r}{\omg(r)^p},\quad 0\le r<1,
\end{split}\end{equation*}
therefore (i) holds.

Next, if (i) holds, then (ii) holds and so it is clear that (iii) holds.
Finally, (iii) together with \eqref{eq:r1} implies (ii). This finishes the proof.
\end{proof}

\end{document}